\documentclass[a4paper]{amsart}

\usepackage[T1]{fontenc}
\usepackage{lmodern}
\usepackage{amssymb}
\usepackage[all]{xy}
\usepackage[inline,shortlabels]{enumitem}
\usepackage{nicefrac,stmaryrd}
\usepackage{mathtools}
\usepackage[british]{babel}
\usepackage{microtype}

\usepackage{tikz-cd, tikz}
\usetikzlibrary{matrix,arrows}
\tikzset{cd/.style=matrix of math nodes,row sep=2em,column sep=2em, text height=1.5ex, text depth=0.5ex}
\tikzset{dar/.style={double,double equal sign distance,-implies}}
\tikzset{mid/.style={anchor=mid}} 
\tikzset{narrowfill/.style={inner sep=0pt, fill=white}}

\usepackage[pdftitle={Topological correspondences},
pdfauthor={Rohit Dilip Holkar},
  pdfsubject={}
]{hyperref}
\usepackage[lite]{amsrefs}
\usepackage{amsfonts}
\usepackage{amsthm}
 \usepackage{nicefrac, color}

\usepackage{caption}

\numberwithin{equation}{section}
\theoremstyle{plain}
\newtheorem{theorem}[equation]{Theorem}
\newtheorem*{theorem*}{Theorem}
\newtheorem{lemma}[equation]{Lemma}
\newtheorem*{lemma*}{Lemma}
\newtheorem{proposition}[equation]{Proposition}
\newtheorem*{proposition*}{Proposition}

\newtheorem{corollary}[equation]{Corollary}
\newtheorem*{corollary*}{Corollary}

\theoremstyle{definition}
\newtheorem{definition}[equation]{Definition}
\newtheorem*{definition*}{Definition}

\theoremstyle{remark}
\newtheorem{remark}[equation]{Remark}

\newtheorem{example}[equation]{Example}
\newtheorem{illustration}[equation]{Illustration}

\newcommand{\Contc}{\mathrm{C}_{\mathrm c}}
\newcommand{\Contz}{\mathrm{C}_0}
\newcommand*{\defeq}{\mathrel{\vcentcolon=}}

\newcommand{\R}{\mathbb R}
\newcommand{\C}{\mathbb C}

\newcommand*{\base}[1][H]{{#1}^{(0)}}

\newcommand{\nb}{\nobreakdash} 

\newcommand{\inverse}{^{-1}}

\newcommand{\CC}{\mathrm C}
\newcommand{\homeo}{\approx}
\newcommand{\iso}{\simeq}

\newcommand{\inpro}[2]{\left\langle#1 \mathbin, #2 \right\rangle}

\newcommand{\tcorr}{\frak{T}}
\newcommand{\ccorr}{\frak{C}}

\newcommand{\htens}{\mathbin{\hat\otimes}}

\newcommand*{\Bound}{\mathbb B}
\newcommand*{\dd}{\mathrm d}

\newcommand*{\Id}{\mathrm{Id}}

\newcommand*{\Cst}{\textup C^*}

\newcommand*{\Hils}[1][H]{\mathcal{#1}}
\newcommand{\Ltwo}{\mathcal L^2}
\newcommand{\bifunct}{\mathfrak{F}}

\title{The bicategory of topological correspondences} \author{Rohit
  Dilip Holkar} \email{rohit.d.holkar@gmail.com} \address{Department
  of Mathematics, Indian Institute of Science Education and Research
  Bhopal, Bhopal Bypass Road, Bhauri, Bhopal 462 066, Madhya Pradesh,
  India.}  \date{\today}

\begin{document}

\subjclass{22D25, 22A22, 47L30, 46L08, 58B30, 46L89, 18D05.}
\keywords{Topological correspondences, bicategory of topological
  correspondences, functoriality of topological correspondences.}
\thanks{This article was partially written during the authors stay as
  an NPDF in IISER Pune which was funded by DST-SERB India. Some of
  the work was also done during the visits to Bhaskaracharya
  Prathisthan Pune. The major part is finished in IISER Bhopal. The
  author is thankful to all these institutes and funding agencies.}

\maketitle{}

\begin{abstract}
  It is known that a topological correspondence \((X,\lambda)\) from a
  locally compact groupoid with a Haar system \((G,\alpha)\) to
  another one, \((H,\beta)\), produces a
  \(\textrm{C}^*\)-correspondence \(\mathcal{H}(X,\lambda)\) from
  \(\textrm{C}^*(G,\alpha)\) to \(\textrm{C}^*(H,\beta)\). In one of
  our earlier article we described composition two topological
  correspondences. In the present article, we prove that second
  countable locally compact Hausdorff topological groupoids with Haar
  systems form a bicategory \(\mathfrak{T}\) when equipped with a
  topological correspondences as 1-arrows. The equivariant
  homeomorphisms of topological correspondences preserving the
  families of measures are the 2-arrows in~\(\mathfrak{T}\).
  
  One the other hand, it well-known that \(\textrm{C}^*\)-algebras
  form a bicateogry \(\mathfrak{C}\) with
  \(\textrm{C}^*\)-correspondences as 1-arrows. The 2-arrows in
  \(\mathfrak{C}\) are unitaries of Hilbert \(\textrm{C}^*\)-modules
  that intertwine the representations. In this article, we show that a
  topological correspondence going to a \(\textrm{C}^*\)-one is a
  bifunctor~\(\mathfrak{T}\to\mathfrak{C}\).
\end{abstract}

\tableofcontents{}

\section{Introduction}

In~\cite{Holkar2017Construction-of-Corr}, we define topological
correspondences and show that a topological correspondence between
locally compact groupoids with Haar systems induce a
\(\Cst\)-correspondence between the groupoid \(\Cst\)\nb-algebras. In
that article, we also give many examples of topological
correspondences most of which are the analogues of the standard
examples of \(\Cst\)\nb-correspondences. These examples show that maps
of spaces (\cite{Holkar2017Construction-of-Corr}{Example~3.1}) and
group homomorphisms (\cite{Holkar2017Construction-of-Corr}*{3.4}) can
be seen as topological correspondences. In the next
article~\cite{Holkar2017Composition-of-Corr} of this series, we
describe how to compose topological correspondences. Examples~4.1
and~4.3 in~\cite{Holkar2017Composition-of-Corr} show that the
composition of space maps and composition of group homomorphisms agree
with the compositions of the topological correspondences associated
with them. Examples related to equivalence of groupoids,
transformation groupoids, induction correspondence are also discussed
these two articles.

Let \((G_i,\alpha_i)\), for \(i=1,2,3\), be locally compact groupoids
with Haar systems, and for \(i=1,2\) let \((X_i,\lambda_i)\) be
topological correspondences from \((G_i,\alpha_i)\) to
\((G_{i+1},\alpha_{i+1})\) with \(X_i\) Hausdorff. Let \(s_{X_1}\) and
\( r_{X_2}\) be the momentum maps for the action of \(G_2\) on \(X_1\)
and \(X_2\), respectively. To define composite of these
correspondences, one considers the proper diagonal action of \(G_2\)
on the fibre product \(X_1\times_{s_{X_1},\base[G_2],
  r_{X_2}}X_2\). Assume that the quotient space
\(X\defeq (X_1\times_{s_{X_1},\base[G_2], r_{X_2}}X_2)/G_2\) is
paracompact. Then we define \emph{a} composite of
\((X_{1},\lambda_{1})\) and \((X_{2},\lambda_{2})\) as a topological
correspondence \((X,\lambda)\colon (G_1,\alpha_1)\to (G_3,\alpha_3)\);
here the family of measures \(\lambda\) is obtained using
\(\lambda_1\) and \(\lambda_2\). Constructing \(\lambda\) is an
involved task, moreover, this family of measures \(\lambda\) is
\emph{not} unique. This family of measures depends on the choice of a
0\nb-cocyle on the transformation groupoid
\( (X_1\times_{s_{X_1},\base[G_2], r_{X_2}}X_2)\rtimes G_2\). The main
result in~\cite{Holkar2017Composition-of-Corr} shows that,
irrespective of choice of \(\lambda\), the \(\Cst\)\nb-correspondence
\(\Hils(X,\lambda)\colon \Cst(G_1,\alpha_2)\to \Cst(G_3,\alpha_3)\) is
isomorphic to the composite \(\Cst\)\nb-correspondence
\(\Hils(X_{2},\lambda_{2})\htens_{\Cst(G_2,\alpha_2)}\Hils(X_{3},\lambda_{3})\). In
present article, we show that the family of measures \(\lambda\) is
unique upto isomorphism of topological correspondences
(Proposition~\ref{prop:lifted-cocycles-are-iso}).  \medskip

A bicategory~\(\mathfrak{S}\) consists of objects, morphisms between
objects called 1\nb-arrows and morphisms between 1\nb-arrows are
called 2\nb-arrows. In literature, 1\nb-arrows are called
1\nb-morphism or simply morphism, and in accordance with this
terminology, a 2\nb-arrow is respectively called a 2\nb-morphism or
bigon. The objects and 1\nb-arrows form a category. Given objects
\(A,B\) in \(\mathfrak{S}\), the class of morphisms \(A\to B\) denoted
by \(\mathfrak{S}(A,B)\) forms a category in which the objects class
is \(\mathfrak{S}(A,B)\) itself and 2\nb-arrows are the
morphisms
. Each object \(A\) in \(\mathfrak{S}\) has the identity 1\nb-arrow
\(I_A\colon A\to A\), and every 1\nb-arrow \(A\xrightarrow{f}B\) has
the identity 2\nb-arrow \(i_{f}\) on it. These identity arrows fulfil
certain identity isomorphisms. The composition of 1\nb-arrows in
\(\mathfrak{S}\) is equipped with the associativity isomorphisms ---
this data is the data required for the a bicategory, and this data
satisfies some coherence conditions, for details
see~\ref{def:bicat}. Bicateogries are also referred to as weak
2\nb-categories.  One may demand that in a bicategory the identity and
associativity \emph{isomorphisms} are replaced by
\emph{equalities}. In this case, the bicategory is called a strict
2\nb-category.

Our main reference for bicategories is B{ \'e}nabou's
notes~\cite{Benabou1967Bicategories}; we follow the terminology and
convention he introduces. Unlike modern authors, B{\'e}nabou writes
arrows the other way round; therefore, so do we.

In~\cite{Buss-Meyer-Zhu2013Higher-twisted}, Buss, Meyer and Zhu define
weak action of a discrete group \(G\) on an object \(A\) in the weak
and twisted 2\nb-categories of \(\Cst\)\nb-algebras; denote these
2\nb-categories by \(\mathfrak{C}\) and \(\mathfrak{M}\),
respectively. And they show that these actions correspond to Fell
bundles over \(G\) and Busby-Smith twisted actions of the group \(G\)
on the \(\Cst\)\nb-algebra \(A\). They also show that the equivalence
of two weak actions in \(\mathfrak{M}\) is same as exterior
equivalence for the Busby-Smith twisted action
(\cite{Busby-Smith1970Rep-twisted-gp-alg}*{Def 2.4}). And that an
equivalence of two actions in \(\mathfrak{C}\) is same as equivalence
of Fell bundles. These result are valid even if \(G\) is a locally
compact group provided that the 2\nb-categories are enriched with
appropriate topological assumptions. This shows that this structural
descriptions of \(\Cst\)-algebras are useful.  \medskip

The interplay between topology or geometry, and operator algebras
---particularly \(\Cst\)\nb-algebras--- has been a topic of deep
interest for mathematicians. The subcategories of the category of
\(\Cst\)\nb-algebras consisting of unital (or non-unital) abelian
\(\Cst\)\nb-algebras equipped with \(*\)\nb-homomorphisms
(\(*\)\nb-homomorphisms into the multiplier algebras, respectively)
are classic classes of \(\Cst\)\nb-algebras which have natural
topological counterparts. We ask a question similar to these classical
questions, namely, if one equips \(\Cst\)\nb-algebras with
\(\Cst\)\nb-correspondences as morphisms, does the new structure have
a topological analogue? If it does, then what is it? How does the
\(\Cst\)\nb-functor behave with these structures?
In~\cite{Buss-Meyer-Zhu2013Higher-twisted}(Section 2.2), Buss, Meyer
and Zhu show that when equipped with \(\Cst\)\nb-correspondences as
morphisms, one may form weak (and strong) bicategory(ies) of
\(\Cst\)\nb-algebras. The 1\nb-arrows are the
\(\Cst\)\nb-correspondences, and the 2\nb-arrows are unitaries of
Hilbert modules intertwining the representations of left
\(\Cst\)\nb-algebras.  Needless to say that this weak bicategory of
\(\Cst\)\nb-algebras \emph{contains} most of the categories of
\(\Cst\)\nb-algebras. With this picture in mind, we look at the
examples at our disposal, which show that the composition of
topological correspondences is well-behaved with the categories of
spaces and groups. Therefore, it is very natural to ask
\begin{enumerate}[(i), leftmargin=*]
\item if topological correspondences yield a categorical structure?
\item If they do then is the assignment that a topological
  correspondence is assigned a \(\Cst\)\nb-one functorial for this
  structure?
\end{enumerate}
We answer both these questions in this article. The answer to the
first one is Theorem~\ref{thm:proof-of-bicategory-of-top-corr} which
states that topological correspondences form a bicategory. And the
answer to the second question is that the \(\Cst\)\nb-assignment is a
bifunctor from the (weak) bicategory of topological correspondences to
that of \(\Cst\)\nb-correspondences which Buss, Meyer and Zhu
describe~in~\cite{Buss-Meyer-Zhu2013Higher-twisted}. In the bicategory
of topological correspondences, the objects are locally compact second
countable groupoids equipped with Haar systems; the 1\nb-arrow are
topological correspondences between groupoids; and measures preserving
equivariant homeomorphisms of topological correspondences are the
2\nb-arrows. The \(\Cst\)\nb-functor maps a groupoid equipped with a
Haar system to its \(\Cst\)\nb-algebra, a topological correspondence
to a \(\Cst\)\nb-correspondence, and 2\nb-arrows to unitary
isomorphisms of \(\Cst\)\nb-correspondences.
 
Proving these fact is not straightforward: In general, composition of
topological correspondences is complicated. This complexity trickles
to the coherence conditions and associativity isomorphism. Therefore,
one has to take care of various minute details involved in the
composition of topological correspondences. The relation between the
invariant families of measures on a proper \(G\)\nb-space and the
families of measures quotient of this space is a key ingredient for
many proofs.  \smallskip

\noindent \emph{Structure of the article:} In the first section,
Section~\ref{sec:recap}, we revise the main results and techniques
from~\cite{Holkar2017Construction-of-Corr}
and~\cite{Holkar2017Composition-of-Corr}. The method of forming a
composite of topological correspondences which appear after
Example~\ref{exa:id-corr-iso} and continue till
Definition~\ref{def:composition} shall be used a lot and it will keep
appearing throughout the article. After that, we re-write the
definition of bicateogry and bifuctor from B{\'e}nabou's
notes~\cite{Benabou1967Bicategories}; this is to establish our
notation and definitions regarding bicategories.

In the second section, we first define an isomorphism of topological
correspondences (Section~\ref{sec:isom-topo-corr}), discuss some
examples of it and prove few useful technical lemmas. Then, in
Section~\ref{sec:bicat-topol-corr}, we define the bicategory of
topological correspondences and prove
Theorem~\ref{thm:proof-of-bicategory-of-top-corr}. In the last part,
in Section~\ref{sec:cst-bifunctor}, we prove
Theorem~\ref{thm:functoriality} which shows that the
\(\Cst\)\nb-assignment is a bifunctor.

We adopt the notation and conventions from the earlier articles in
this series. We believe that the many examples from the earlier two
articles suffice for this one as well. At the end, we give three
illustrations using Examples~3.1
in~\cite{Holkar2017Construction-of-Corr} and~4.1
in~\cite{Holkar2017Composition-of-Corr}, Examples~3.3
in~\cite{Holkar2017Construction-of-Corr} and~4.2
in~\cite{Holkar2017Composition-of-Corr}, and Examples~3.4
in~\cite{Holkar2017Construction-of-Corr} and~4.3
in~\cite{Holkar2017Composition-of-Corr}.

\section{Recap}
\label{sec:recap}

\subsection{Topological correspondences}

Let \(G\) be a groupoid. Then \(\base[G]\) denotes the set of units of
\(G\).  Let \(X\) be a left (or right) \(G\)\nb-space; we tacitly
assume that \(r_X\) (respectively, \(s_X\)) is the momentum map for
the action. The transformation groupoid for this action is denoted by
\(G\ltimes X\) (respectively, \(X\rtimes G\)). By \(r_G\) and \(s_G\)
we denote the range and the source maps of \(G\), respectively, which
are also the momentum maps for the left and right multiplicatin action
of \(G\) on itself. The fibre product \(G\times_{s_G,\base[G],r_X}X\)
of \(G\) and \(X\) over \(\base[G]\) along \(s_G\) and \(r_X\) is
denoted by \(G\times_{\base[G]}X\). If \(X\) is a right
\(G\)\nb-space, then \(X\times_{\base[G]}G\) has a similar meaning. If
\(X\) and \(Y\) are, respectively, left and right \(G\) spaces, then
we denote the fibre product \(X\times_{s_X,\base[G],r_Y}Y\) by
\(X\times_{\base[G]}Y\), that is,
\(X\times_{\base[G]}Y=\{(x,y)\in X\times Y:s_X(x)=r_Y(y)\}\).

Let \(A,B\) be \(\Cst\) algebras, \(\Hils\) a Hilbert \(B\)\nb-module
and \(\phi\colon A\to \Bound(\Hils)\) a nondegenerate representation
that makes \((\Hils,\phi)\) a \(\Cst\)\nb-correspondence from \(A\) to
\(B\). Then we simply call \(\Hils\) a \(\Cst\)\nb-correspondence from
\(A\) to \(B\). Let \(C\) be another \(\Cst\)\nb-algebra and
\(\Hils[K]\colon B\to C\) a \(\Cst\)\nb-correspondence. Then
\(\Hils\htens_B \Hils[K]\) is the interior tensor product of the
Hilbert modules; we may also write \(\Hils\htens \Hils[K]\) when the
middle \(\Cst\)\nb-algebra is clear.

By \(\R^*_+\) we denote the multiplicative group of positive real
numbers.

\begin{definition}[Topological correspondence
  (\cite{Holkar2017Construction-of-Corr} Definition 2.1)]
  \label{def:correspondence}
  A \emph{topological correspondence} from a locally compact groupoid
  \(G\) with a Haar system \(\alpha\) to a locally compact groupoid
  \(H\) equipped with a Haar system \(\beta\) is a pair
  \((X, \lambda)\), where:
  \begin{enumerate}[label={\roman*}),leftmargin=*] \item \(X\) is a
    locally compact \(G\)-\(H\)-bispace,
  \item the action of \(H\) is proper,
  \item \(\lambda = \{\lambda_u\}_{u\in\base}\) is an
    \(H\)\nb-invariant continuous family of measures along the
    momentum map \(s_X\colon X\to\base\),
  \item \(\Delta\) is a continuous function
    \(\Delta: G \ltimes X \rightarrow \R^+\) such that for each
    \(u \in \base\) and \(F\in \Contc(G\times_{\base[G]} X)\),
    \begin{multline*}
      \int_{X_{u}} \int_{G^{r_X(x)}} F(\gamma\inverse, x)\,
      \dd\alpha^{r_X(x)}(\gamma)\, \dd\lambda_{u}(x) \\= \int_{X_{u}}
      \int_{G^{r_X(x)}} F(\gamma, \gamma\inverse x)\, \Delta(\gamma,
      \gamma\inverse x) \, \dd\alpha^{r_X(x)}(\gamma)
      \,\dd\lambda_{u}(x).
    \end{multline*}
  \end{enumerate}
\end{definition}
The function \(\Delta\) is unique and is called \emph{the adjoining
  function} of the correspondence. The family of measures~\(\lambda\)
above is called an \(s_X\)\nb-system
in~\cite{Renault1985Representations-of-crossed-product-of-gpd-Cst-Alg}.

For \(\phi \in \Contc(G)\), \(f \in \Contc(X)\) and
\(\psi \in \Contc(H)\) define the functions \(\phi\cdot f\) and
\(f\cdot \psi\) on \(X\) as
follows: \begin{equation}\label{def:left-right-action}
  \left\{\begin{aligned} (\phi\cdot f)(x) &\defeq \int_{G^{r_X(x)}}
      \phi(\gamma) f(\gamma\inverse x) \,
      \Delta^{1/2}(\gamma, \gamma\inverse x) \; \dd \alpha^{r_X(x)}(\gamma),\\
      (f\cdot\psi )(x) &\defeq \int_{H^{s_X(x)}}
      f(x\eta)\psi({\eta}\inverse) \; \dd \beta^{s_X(x)}(\eta).
    \end{aligned}\right.
\end{equation}
For \(f, g \in \Contc(X)\) define the function
\(\langle f, g \rangle \) on \(H\) by
\begin{align}\label{def:inner-product}
  \langle f, g \rangle (\eta) &\defeq \int_{X_{r_H(\eta)}} \overline{f(x)} g(x\eta) \; \dd
                                \lambda_{r_H(\eta)}(x).
\end{align}
We often write \(\phi f\) and \(f\psi\) instead of \(\phi\cdot f\) and
\(f\cdot \psi\), respectively.  Lemma
in~\cite{Holkar2017Construction-of-Corr} shows that
\(\phi f, f\psi\in \Contc(X)\) and \(\inpro{f}{g}\in \Contc(H)\).

\begin{theorem}[\cite{Holkar2017Construction-of-Corr}*{Theorem
    2.10}]\label{thm:mcorr-gives-ccorr}
  Let \((G, \alpha)\) and \((H, \beta)\) be locally compact groupoids
  with Haar systems. Then a topological correspondence
  \((X, \lambda)\) from \((G, \alpha)\) to \((H, \beta)\) produces a\;
  \(\Cst\)\nb-correspondence \(\Hils(X,\lambda)\) from
  \(\Cst(G,\alpha)\) to \(\Cst(H,\beta)\).
\end{theorem}

\begin{example}[The identity topological correspondence]
  \label{exa:id-corr-iso}
  Let \((G,\alpha)\) be a locally compact topological groupoid with a
  Haar system. Then define the space \(X=G\). Then \(X\) is a free, as
  well as, proper \(G\)-\(G\)\nb-bispace with the left and right
  multiplication actions of \(G\) on itself. Example~3.7
  in~\cite{Holkar2017Construction-of-Corr} shows that \(G\) is a Macho
  Stadler--O'uchi correspondence from \(G\) to itself; in fact \(X\)
  is an equivalence of \(G\) with itself. What is the family of
  measures \(\lambda\) on \(X\) along the right momentum map \(s_G\)
  that makes it a topological correspondence in the sense of
  Definition~\ref{def:correspondence}? The discussion
  in~\cite{Holkar2017Construction-of-Corr}*{Example~3.7} shows that
  for \(k\in\Contc(X)\) and \(u\in\base[G]\)
  \[
    \int_X k\,\lambda_u\defeq \int_G k(\gamma\inverse
    x)\,\dd\alpha^{r_G(x)}(\gamma)
  \]
  where \(x\in X\) is any element with \(s_G(x)=u\). Since
  \(\lambda_u\) does not depend on \(x\in s_G\inverse(u)\), we choose
  \(x=u\in \base[G]\) which shows that
  \[
    \int_X k\,\lambda_u\defeq \int_G
    k(\gamma\inverse)\,\dd\alpha^{r_G(x)}(\gamma).
  \]
  Thus \(\lambda=\alpha\inverse\). The adjoining function for this
  correspondences is the constant function 1.

  Moreover, \(\Hils(X,\alpha\inverse)\) and \(\Cst(G,\alpha)\) are
  same \emph{as} Hilbert \(\Cst(G,\alpha)\)\nb-module, and the
  isomorphism is implemented by the identity map
  \(\Id_G\colon X\to G\). To see this, we firstly notice that
  \(\Contc(X,\beta\inverse)\) is a dense complex vector subspace of
  \(\Cst(G,\alpha)\), as well as, \(\Hils(X,\alpha\inverse)\). For
  \(f\in \Contc(X)\) and \(\psi\in\Contc(G)\),
  Equation~\ref{def:left-right-action} gives us
  \begin{equation}
    f\cdot \psi(x)=\int_G
    f(x\eta)\psi(\eta\inverse)\,\dd\alpha^{s_G(x)}(\eta) =
    f*\psi(x)\label{eq:id-corr-iso}
  \end{equation}
  where \(x\in X\), and \(f*\psi\) is the convolution of
  \(f,\psi\in\Contc(G)\subseteq \Cst(G,\alpha)\). If
  \(g\in \Contc(X)\) is another function and \(\eta\in G\), then
  Equation~\ref{def:inner-product} says
  \[
    \inpro{f}{g}(\gamma)=\int_G \overline{f(x)}
    g(x\gamma)\,\dd\alpha\inverse_{r_G(\gamma)}(x)
  \]
  which equals
  \[
    \int_G \overline{f(x\inverse)}
    g(x\inverse\gamma)\,\dd\alpha^{r_G(\gamma)}(x)=\int_G f^*(x)
    g(x\inverse\gamma)\,\dd\alpha^{r_G(\gamma)}(x)= f^**g(\eta)
  \]
  where \(f^*\) is the involution of
  \(f\in \Contc(G)\subseteq \Cst(G,\alpha)\) and \(f^**g\) denotes the
  convolution as earlier. From the construction of
  \(\Hils(X,\alpha\inverse)\) (proof of
  Theorem~\ref{thm:mcorr-gives-ccorr}), it is clear that
  \(\Hils(X,\alpha\inverse)=\Cst(G,\alpha)\) as Hilbert
  \(\Cst(G,\alpha)\)\nb-modules. Finally, as in
  Equation~\eqref{eq:id-corr-iso} above, one can show that
  \(\psi\cdot f=\psi* f\) which shows that \(\Hils(X,\alpha\inverse)\)
  and \(\Cst(G,\alpha)\) are \(\Cst\)\nb-correspondences on
  \(\Cst(G,\alpha)\).
\end{example}
\medskip

Let \((G_i,\chi_i)\) be a locally compact groupoid with a Haar system
for \(i=1,2,3\). Let \((X,\alpha)\colon (G_1,\chi_1)\to (G_2,\chi_2)\)
and \((Y,\beta)\colon (G_2,\chi_2)\to (G_3,\chi_3)\) be topological
correspondences, and let \(\Delta_1\) and \(\Delta_2\) be their
adjoining functions, respectively. Additionally, assume that \(X\) and
\(Y\) are Hausdorff, and \((X\times_{\base[G_2]}Y)/G_2\) is
paracompact. Recall from page~\pageref{sec:recap} that
\(X\times_{\base[G_2]}Y\) denotes the fibre product
\(X\times_{s_X,\base[G_2],r_Y}Y\). Now,
from~\cite{Holkar2017Composition-of-Corr}, we recall how to form the
composite \((Y,\beta)\circ(X,\alpha)\): We need to find a
\(G_1\)\nb-\(G_2\)-bispace \(\Omega\) (obtained using \(X\) and \(Y\))
and a family of measures \(\mu=\{\mu_u\}_{u\in\base[G]_3}\) (obtained
from \(\alpha\) and \(\beta\)) with the properties that (i)
\((\Omega,\mu)\) is a correspondence from \((G_1,\chi_1)\) to
\((G_3,\chi_2)\) and (ii) we have an isomorphism of
\(\Cst\)\nb-correspondences
\(\Hils(\Omega,\mu)\iso\Hils(X,\alpha)\htens_{\Cst(G_2,\chi_2)}\Hils(Y,\beta)\).
Reader may refer~\cite{Holkar2017Composition-of-Corr}*{Section 2.1}
or~\cite{Lance1995Hilbert-modules}*{Chapter 4} for details of interior
tensor product of Hilbert \(\Cst\)\nb-modules and composition of
\(\Cst\)\nb-correspondences.

Continuing the above discussion, let \(Z\) denote the fibre product
\(X\times_{\base[G_2]}Y\); then \(Z\) is a \(G_1\)-\(G_2\)\nb-bispace
with the obvious left and right actions. Let
\(s_Z\colon Z\to \base[G]_2\) be the map \(s_Z(x,y)=s_X(x)\),
equivalently, \(s_Z(x,y)=r_Y(y)\). The space \(Z\) carries the
diagonal action of \(G_2\), that is,
\((x,y)\gamma=(x\gamma,\gamma\inverse y )\) for
\((x,y,\gamma)\in Z\times_{\base[G_2]}G_2\); \(s_Z\) is the momentum
map for this action. Since the action of \(G_2\) on \(X\) is proper
(by hypothesis), so is that of \(G_2\) on \(Z\). We define
\(\Omega=Z/G_2\). the quotient space \(\Omega\) is a
\(G_1\)-\(G_3\)\nb-bispace with the actions induced from those on
\(Z\). Moreover, the right action of \(G_3\) in \(\Omega\) is proper
(\cite{Holkar2017Composition-of-Corr}*{Lemma 3.4}). This \(\Omega\) is
the desired space in the composite. Now we form the composite of the
families of measures on \(\Omega\).

\noindent (1)\label{page:compo-rev-1} Fix \(u\in\base[G_3]\). Define
the measure \(m_u\) on the space \(Z\) by
\[
  \int_{Z} f \;\dd m_u = \int_Y\int_X f(x,y) \;
  \dd\alpha_{r_Y(y)}(x)\; \dd\beta_{u}(y)
\]
where \(f\in \Contc(Z)\).

\noindent (2)\label{page:compo-rev-2} Let \(\pi\colon Z\to \Omega\) be
the quotient map. For \(f \in \Contc(Z)\) and
\(\omega=[x,y]\in\Omega\), let
\[
  \int_Z f\; \dd\lambda^{\omega} \defeq \int_{G_2^{r_Y(y)}} f(x\gamma,
  \gamma\inverse y) \; \dd\chi_2^{r_Y(y)}(\gamma);
\]
notice that \(\lambda^\omega\) is, in fact, defined over
\({\pi\inverse(\omega)}\subseteq Z\).  Then
\(\lambda=\{\lambda^\omega\}_{\omega\in \Omega}\) is a continuous
family of measures along \(\pi\).

\noindent (3) \label{page:compo-rev-3} It is well-known that the Haar
system \(\chi_2\) of \(G_2\) induce a Haar system \(\chi\) on the
transformation groupoid \(Z\rtimes G_2\): for
\(f\in \Contc(Z\rtimes G_2)\) and \(v\in Z\),
\[
  \int_{Z\rtimes G_2} f\,\dd\chi^v\defeq \int_{G_2} f(\gamma\inverse,
  v)\,\dd\chi_2^{s_Z(v)}(\gamma).
\]
Let \(\chi\inverse\) be the corresponding right invariant Haar system
on \(Z\rtimes G_2\), that is,
\(\int_{Z\rtimes G_2} f\,\dd\chi\inverse_v=\int_{Z\rtimes G_2}
f\circ\textup{inv}_{Z\rtimes G_2}\,\dd\chi^v\) for
\(f\in \Contc(Z\rtimes G_2)\) and \(v\in \base[(Z\rtimes G_2)]\). Here
\(\textup{inv}_G\colon G\to G\) is the homeomorphism
\(\textup{inv}_G(\gamma)=\gamma\inverse\) for \(\gamma\in G\).
Figure~\ref{fig:pushig-measure-down} shows the maps in (1)--(3) and
the families of measures along with them.
\begin{figure}[htb]
  \centering
  \[\begin{tikzcd}[column sep=huge,row sep=normal]
      Z\rtimes G_2 \arrow{r}{\chi\inverse}[swap]{s_{Z\rtimes G_2}}
      \arrow{d}[swap]{\chi}{r_{Z\rtimes G_2}}
      & Z \arrow{d}{\lambda}[swap]{\pi} \\
      Z \arrow{r}{\lambda}[swap] {\pi} & \Omega
    \end{tikzcd}\]
  \caption{}
  \label{fig:pushig-measure-down}
\end{figure}

The first part in the proof of Lemma~{3.6}
in~\cite{Holkar2017Composition-of-Corr} shows that, for each
\(u\in Z\), the measure \(m_u\) on \(Z=\base[Z\rtimes G_2]\) is
\((Z\rtimes G_2, \chi)\)\nb-quasi-invariant. That is, there is an
\(\R^*_+\)\nb-valued continuous 1\nb-cocycle~\(D\) (denoted by
\(\Delta\) in~\cite{Holkar2017Composition-of-Corr})
on~\(Z\rtimes G_2\) with the property that
\(m_u\circ\chi=D\, (m_u\circ\chi\inverse)\). The cocycle \(D\) is
given by
\begin{equation}
  \label{equ:composition-1-cocycle}
  \Delta\colon ((x,y),\gamma)\mapsto\Delta_2(\gamma\inverse,
  y).
\end{equation}

Recall from the discussion a few paragraphs above, that the action of
\(G_2\) on \(Z\) is proper, that is, the transformation groupoid
\(Z\rtimes G_2\) is proper. Proposition~{2.7}
in~\cite{Holkar2017Composition-of-Corr} says that every
\(\R\)\nb-valued 1\nb-cocycle on \(Z\rtimes G_2\) is a
coboundary. Using this, we get a family of 0\nb-cocycles
\(b=\{b_u\}_{u\in Z}\) such that
\(D=\frac{b_u\circ s_{Z\rtimes G_2}}{b_u\circ r_{Z\rtimes G_2}}\)
for each \(u\in Z\). Now~\cite{Holkar2017Composition-of-Corr}*{Lemma
  2.7} shows that \(b_u m_u\) is a \((Z\rtimes G_2,\chi)\)-symmetric
measure on its space of units, that is,
\((b_um_u)\circ\chi= (b_um_u)\circ\chi\inverse\). Fix \(u\in Z\).
Now~\cite{Holkar2017Composition-of-Corr}*{Proposition 3.1} says that
there is a measure \(\mu_u\) on \(\Omega\) which gives the
disintegration of measures \(b_um_u\circ\chi = \mu_u\). We write \(b\)
instead of \(b_u\). Then this measure is given by
\begin{equation}
  \label{eq:composite-measure}
  \int_\Omega f[x,y]\,\dd\mu_u'([x,y]) =\int_Y\int_X f\circ\pi(x,y) e(x,y)\, b(x,y)\, \dd\alpha_{r_Y(y)}(x)\, \dd\beta_u(y)
\end{equation}
for \(f\in \Contc(\Omega)\). In the above equation,
\(\pi\colon Z\to \Omega\) is the quotient map, \(e\colon Z\to \R^+\) a
cutoff function; see~Proposition~{3.1(ii)}
in~\cite{Holkar2017Composition-of-Corr}.

Finally,~\cite{Holkar2017Composition-of-Corr}*{Proposition 3.10} shows
that \(\mu=\{\mu_u\}_{u\in G_3}\) is a \(G_3\)\nb-invariant continuous
family of measures on \(\Omega\).
And~\cite{Holkar2017Composition-of-Corr}*{Proposition 3.12} shows that
each~\(\mu_u\) is \((G_1,\chi_1)\)\nb-quasi-invariant; the adjoining
function is
\[
  \Delta_{1,2}(\eta,[x,y])\defeq b(\eta x,y)\inverse \Delta_1(\eta,x)
  b(x,y)
\]
where \(\Delta_1\colon G_1\ltimes X\to \R_+^*\) is the adjoining
function of the topological correspondence \((X,\alpha)\).

\begin{definition}[Composite]
  \label{def:composition}
  Let
  \[
    (X, \alpha)\colon (G_1,\chi_1) \rightarrow (G_2,\chi_2)
    \quad\text{ and } \quad (Y, \beta)\colon (G_2,\chi_2) \rightarrow
    (G_3,\chi_3)
  \]
  be topological correspondences with \(\Delta_1\) and \(\Delta_2\) as
  the adjoining function, respectively. A composite of these
  correspondences
  \((\Omega, \mu):(G_1,\chi_1)\rightarrow (G_3,\chi_3)\) is defined
  by:
  \begin{enumerate}[label= {\roman*)},leftmargin=* ]
  \item the space \(\Omega \defeq (X\times_{\base[G_2]}Y)/ G_2\),
  \item a family of measures \(\mu = \{ \mu_u\}_{u\in{\base[G]}_3}\)
    that lifts to \(\{b(\alpha \times \beta_u)\}_{u\in{\base[G_3]}}\)
    on \(Z\) for a cochain
    \(b\in \CC^0_{G_3}((X\times_{\base[G_2]}Y)\rtimes G_2, \R_+^*)\)
    satisfying \(d^0(b)=D\) where
    \(D\colon (X\times_{\base[G_2]}Y)\rtimes G_2\to \R^*_+\) is
    \(D((x,y),\gamma)=\Delta_2(\gamma\inverse, y)\).
  \end{enumerate} \end{definition} \begin{theorem}[Theorem 3.14,
  \cite{Holkar2017Composition-of-Corr}]
  \label{thm:well-behaviour-of-composition}
  Let
  \[(X, \alpha)\colon (G_1,\chi_1) \rightarrow (G_2,\chi_2) \quad
    \text{and} \quad (Y, \beta) \colon (G_2,\chi_2) \rightarrow
    (G_3,\chi_3)\] be topological correspondences of locally compact
  groupoids with Haar systems. In addition, assume that \(X\) and
  \(Y\) are Hausdorff and second countable. Let
  \((\Omega, \mu)\colon (G_1,\chi_1)\rightarrow (G_3,\chi_3)\) be a
  composite of the correspondences. Then \(\Hils(\Omega,\mu) \) and
  \(\Hils(X,\alpha)\, \hat{\otimes}_{\Cst(G_2,\chi_2)}
  \Hils(Y,\beta)\) are isomorphic \(\Cst\)\nb-correspondences from
  \(\Cst(G_1,\chi_1)\) to \(\Cst(G_3,\chi_3)\). \end{theorem}

\subsection{Bicategory}

We follow B\'enabou's notation and terminology,
from~\cite{Benabou1967Bicategories}, for bicategories. B\'enabou's
convention for composition is the other way round than the standard
one. As Leinster shows in~\cite{Leinster1998Bicategories}, a
bicategory is biequivalent to a \(2\)\nb-category.

\begin{definition}[Bicategory, \cite{Benabou1967Bicategories}*{Definition 1.1}]\label{def:bicat}
  A bicategory \(\mathfrak{S}\) is determined by the following data:
  \begin{enumerate}[label=\roman*), leftmargin=*]
  \item a set \(\mathfrak{S}_0\) called set of objects or vertices;
  \item for each pair \((A,B)\) of objects, a category
    \(\mathfrak{S}(A,B)\);
  \item for each triple \((A,B,C)\) of objects of \(\mathfrak{S}\) a
    \emph{composition functor}
    \[
      c(A,B,C)\colon \mathfrak{S}(A,B) \times \mathfrak{S}(B,C)
      \rightarrow \mathfrak{S}(A,C) ;\]
  \item for each object \(A\) of \(\mathfrak{S}\) an object \(I_A\) of
    \(\mathfrak{S}(A,A)\) called \emph{identity arrow} of~\(A\) (the
    identity map of \(I_A\) in \(\mathfrak{S}(A,A)\) is denoted
    \(i_A\colon I_A\implies I_A\) and is called \emph{identity 2-cell}
    of \(A\));
  \item for each quadruple \((A,B,C,D)\) of objects of
    \(\mathfrak{S}\), a natural isomorphism \(a(A,B,C,D)\) called
    \emph{associativity isomorphism} between the two composite
    functors making the following diagram commute:
    \begin{center}
      \begin{tikzpicture}[scale=1]
        \draw[->] (1.7,0)--(5.2,0); \draw[->] (0,1.7)--(0,0.2);
        \draw[->](2.5,2)--(4.4,2); \draw[->] (5.8, 1.7)--(5.8,0.24);
        \draw[->, double] (1.8, 0.6)--(4.2,1.4); \node at (6, 0)[
        scale=1]{\(\mathfrak{S}(A,D)\)}; \node at (6.2, 2)[
        scale=1]{\(\mathfrak{S}(A,B)\times \mathfrak{S}(B,D)\)}; \node
        at (-0.2, 2)[
        scale=1]{\(\mathfrak{S}(A,B)\times \mathfrak{S}(B,C)\times
          \mathfrak{S}(C,D)\)}; \node at
        (2.7,1.1)[scale=0.8]{\(\sim\)}; \node at
        (3.8,0.8)[scale=0.8]{\(a(A,B,C,D)\)}; \node at (0, 0)[
        scale=1]{\(\mathfrak{S}(A,C)\times \mathfrak{S}(C,D)\)}; \node
        at (3.45, 2.25)[scale=0.8]{\(\textup{Id}\times c(B,C,D)\)};
        \node at (-1.1,
        1.0)[scale=0.8]{\( c(A,B,C)\times \textup{Id}\)}; \node at
        (3.1, -0.2)[scale=0.8]{\(c(A,C,D)\)}; \node at (6.6,
        1.0)[scale=0.8]{\(c(A,B,D)\)};
      \end{tikzpicture}
    \end{center}


  \item for each pair \((A,B)\) of objects of \(\mathfrak{S}\), two
    natural isomorphisms \(l(A,B)\) and \(r(A,B)\), called left and
    right identities such that the following diagrams commute:
    \begin{center}
      \begin{tikzpicture}[scale=1]
        \draw[->] (1.1,2)--(2.3,2); \draw[->] (0.2,1.8)--(1.7,0.3);
        \draw[->](3.8,1.8)--(2.3,0.3); \draw[->, double]
        (1.2,1.2)--(1.8,1.8); \node at
        (0,2)[scale=1]{\(1\times\mathfrak{S}(A,B)\)}; \node at
        (4,2)[scale=1]{\(\mathfrak{S}(A,A)\times \mathfrak{S}(A,B)\)};
        \node at (2,0){\(\mathfrak{S}(A,B)\)}; \node at
        (2,1.3)[scale=0.8]{\(l(A,B)\)}; \node at (1.7
        ,2.2)[scale=0.8]{\(\textup{I}_A\times\textup{Id}\)}; \node at
        (0.3, 0.9)[scale=0.8]{canonical}; \node at
        (1.3,1.0)[scale=0.8]{\(\sim\)}; \node at (4,
        1.1)[scale=0.8]{\(c(A,A,B)\)};
      \end{tikzpicture}
    \end{center}
    \begin{center}
      \begin{tikzpicture}[scale=1]
        \draw[->] (1,2)--(2.3,2); \draw[->] (0.3,1.8)--(1.6,0.3);
        \draw[->](3.8,1.8)--(2.3,0.3); \draw[->, double]
        (1.2,1.2)--(1.8,1.8); \node at
        (0,2)[scale=1]{\(\mathfrak{S}(A,B)\times 1\)}; \node at
        (4,2)[scale=1]{\(\mathfrak{S}(A,B)\times \mathfrak{S}(B,B)\)};
        \node at (2,0){\(\mathfrak{S}(A,B)\)}; \node at
        (2,1.3)[scale=0.8]{\(r(A,B)\)}; \node at (1.7
        ,2.2)[scale=0.8]{\(\textup{Id}\times\textup{I}_B\)}; \node at
        (0.3, 0.9)[scale=0.8]{canonical}; \node at
        (1.3,1.0)[scale=0.8]{\(\sim\)}; \node at (4,
        1.1)[scale=0.8]{\(c(A,B,B)\)};
      \end{tikzpicture}
    \end{center}

    This data satisfies the following conditions:
  \item \textit{associativity coherence}: If \((S,T,U,V)\) is an
    object of
    \(\mathfrak{S}(A,B)\times \mathfrak{S}(B,C)\times
    \mathfrak{S}(C,D)\times \mathfrak{S}(D,E)\), then the following
    diagram commutes:
    \begin{center}
      \begin{tikzpicture}[scale=0.9]
        \draw[dar] (0,1.8)--(0,0.2); \draw[dar] (0.1, -0.3)--(2,-1.5);
        \draw[dar] (5.8,-0.2)--(3.6,-1.4); \draw[dar] (5.8,
        1.8)--(5.8,0.24); \draw[dar](1.6,2)--(4.4,2);
        \node at (6, 0)[ scale=1]{\(S\circ((T\circ U)\circ V)\)};
        \node at (6, 2)[ scale=1]{\((S\circ (T\circ U))\circ V\)};
        \node at (0, 2)[ scale=1]{\(((S\circ T)\circ U)\circ V\)};
        \node at (3, -1.7)[scale=1]{\( S\circ(T\circ(U\circ V))\)};
        \node at ((-1.2, 1)[scale=0.8]{\(a(S\circ T, U, V)\)}; \node
        at (0, 0)[ scale=1]{\((S\circ T)\circ (U\circ V)\)}; \node at
        (3.0, 2.25)[scale=0.8]{\(a(S,T,U)\circ\textup{Id}_V\)}; \node
        at (-0.2, -1.0)[scale=0.8]{\(a(S,T, U\circ V)\)}; \node at
        (5.8, -1.0)[scale=0.8]{\(\textup{Id}_S\circ a(T, U, V)\)};
        \node at (7, 1.0)[scale=0.8]{\(a(S, T\circ U, V)\)};
      \end{tikzpicture}
    \end{center}

  \item \emph{identity coherence}: If \((S,T)\) is an object of
    \(\mathfrak{S}(A,B)\times \mathfrak{S}(B,C)\), then the following
    diagram commutes:
    \begin{center}
      \begin{tikzpicture}[scale=1]
        \draw[dar] (1,2)--(3,2); \draw[dar] (0.2,1.8)--(1.7,0.3);
        \draw[dar](3.8,1.8)--(2.3,0.3); \node at
        (0,2)[scale=1]{\((S\circ\textup{I}_B)\circ T\)}; \node at
        (4,2)[scale=1]{\(S\circ (\textup{I}_B\circ T)\)}; \node at
        (2,0){\(S\circ T\)}; \node at (1.9,
        2.2)[scale=0.8]{\(a(S, I_B, T)\)}; \node at (0.3,
        0.9)[scale=0.8]{\(r(S)\circ \textup{Id}_T\)}; \node at (4,
        1.1)[scale=0.8]{\(\textup{Id}_S\circ l(T)\)};
      \end{tikzpicture}
    \end{center}
  \end{enumerate}
\end{definition}

In modern literature, a vertex, an arrow (or a 1\nb-cell) and a 2-cell
are called an object, a 1\nb-arrow and a 2\nb-arrow, respectively. Let
\(A\) and \(B\) be two objects and let \(t,u\) be two arrows in the
category \(\mathfrak{S}(A,B)\). Then we call the rule of composition
of \(t\) and \(u\) in \(\mathfrak{S}(A,B)\) the vertical composition
of 1\nb-arrows. The composite functor \(c\) in (iii) above gives the
horizontal composition of 2\nb-arrows.  Let \((S,T)\) and \((S',T')\)
be two objects in \(\mathfrak{S}(A,B)\times\mathfrak{S}(B,C)\),
respectively, and let \(s\colon S\to S'\) and \(t\colon T\to T'\) be
2\nb-arrows. Then \(s\) and \(t\) induce a 2\nb-arrow
\(s\cdot_{h}t\colon S\circ T\to S'\circ T'\). The 2\nb-arrow
\(s\cdot_{h}t\) is called the vertical composite of the 2\nb-arrows
\(s\) and \(t\).

\begin{example}
  \label{exa:Cst-corr}
  In Section 2.2 of~\cite{Buss-Meyer-Zhu2013Higher-twisted} Buss,
  Meyer and Zhu form a bicategory of \(\Cst\)\nb-algebraic
  correspondences. In this bicategory the objects are the
  \(\Cst\)\nb-algebras, 1-arrows are the \(\Cst\)\nb-algebraic
  correspondences and 2-arrows are the equivariant unitary
  intertwiners of \(\Cst\)\nb-correspondences.
\end{example}

\begin{example}
  The \(\Cst\)\nb-correspondences of commutative (or commutative and
  unital) \(\Cst\)\nb-algebras is a sub-bicategory of the bicategory
  in Example~\ref{exa:Cst-corr}.
\end{example}

\begin{definition}[Morphisms of bicategories,
  \cite{Benabou1967Bicategories}*{Definition
    4.1}]\label{def:bifunctor}
  Let \(\mathfrak{S}\) and \(\mathfrak{S}'\) be bicategories. A
  morphism \(\mathfrak{V}=(V, v)\) from \(\mathfrak{S}\) to
  \(\mathfrak{S}'\) consists of:
  \begin{enumerate}[label=\roman*), leftmargin=*]
  \item a map \(V\colon \mathfrak{S}_0 \rightarrow\mathfrak{S}'_0\)
    sending an object \(A\) to \(V(A)\);
  \item a family of functors
    \(V(A,B)\colon \mathfrak{S}(A,B)\rightarrow \mathfrak{S}'(V(A),
    V(B))\) sending a 1-cell \(S\) to \(V(A)\) and a 2-cell \(s\) to
    \(V(s)\);
  \item for each object \(A\) of \(\mathfrak{S}\), a 2-cell
    \(v_A\in \mathfrak{S}(V(A),V(B))\)
    \[v_A\colon I_{V(A)}\Rightarrow V(I_A);\]
  \item a family of natural transformations
    \[
      v(A,B,C)\colon c(V(A),V(B),V(C))\circ(V(A,B)\times
      V(B,C))\rightarrow V(A,C)\circ c(A,B,C).
    \]
    If \((S,T)\) is an object of
    \(\mathfrak{S}(A,B)\times\mathfrak{S}'(B,C)\), the
    \((S,T)\)\nb-components of \(v(A,B,C)\)
    \[
      v(A,B,C)(S,T)\colon V(S)\circ V(T)\Rightarrow V(S\circ T)
    \]
    shall be abbreviated \(v\) or \(v(S,T)\).

    This data satisfies the following coherence conditions:
  \item If \((S,T,U)\) is an object of
    \(\mathfrak{S}(A,B)\times\mathfrak{S}(B,C)\times\mathfrak{S}(C,D)\)
    the diagram in Figure~\ref{fig:actual-pentagon} is commutative.

 \begin{figure}[ht]
   \centering
   \begin{tikzpicture}[scale=2]
     \draw[dar] (0,0.9)--(0,0.1);
     \draw[dar] (0,1.9)--(0,1.1);
     \draw[dar] (2.2,2)--(0.8,2); 
     \draw[dar] (3,1.9)--(3,1.1);
     \draw[dar] (3, 0.9)--(3,0.1);
     \draw[dar] (2.4,0)--(0.5, 0);
     \node at (0,0 )[scale=0.8]{\(V(S\circ(T\circ U))\)}; \node at
     (0,1 )[scale=0.8]{\(V(S)\circ V(T\circ U)\)}; \node at (0,2
     )[scale=0.8]{\(V(S)\circ(V(T)\circ V(U))\)}; \node at (3,2
     )[scale=0.8]{\((V(S)\circ V(T))\circ V(U)\)}; \node at
     (3,1)[scale=0.8]{\(V(S\circ T)\circ V(U)\)}; \node at (3,0
     )[scale=0.8]{\(V((S\circ T)\circ U)\)};
     \node at (-0.4,0.5 )[scale=0.8]{\(v(S,T\circ U)\)}; \node at
     (-0.55,1.5 )[scale=0.8]{\(\textup{Id}_{V(S)}\circ v(T,U)\)};
     \node at (1.5,2.1 )[scale=0.8]{\(a(V(S),V(T),V(U))\)}; \node at
     (1.5,1.9 )[scale=0.8]{\(\sim\)}; \node at (3.55,1.5
     )[scale=0.8]{\(v(S,T)\circ\textup{Id}_{V(U)}\)}; \node at
     (3.4,0.5 )[scale=0.8]{\(v(S\circ T, U)\)}; \node at (1.5,-0.1
     )[scale=0.8]{\(V(a(S, T, U))\)}; \node at (1.5,0.1
     )[scale=0.8]{\(\sim\)};
   \end{tikzpicture}
   \caption{\small Associativity coherence for a transformation
     between bicategories}
   \label{fig:actual-pentagon}
 \end{figure}

\item If \(S\) is an object of \(\mathfrak{S}(A,B)\) then the diagram
  in Figure~\ref{fig:right-identity-coherence}, for the right identity
  commutes.

  \begin{figure}[hbt]
    \[
      \begin{tikzcd}[scale=2]
        V(S)
        &V(S\circ \mathit{I}_B) \arrow[l,Rightarrow]{v_r}[swap]{\sim} \\
        V(S)\circ \textup{I}_{V(B)} \arrow[r,
        Rightarrow]{}[swap]{\textup{Id}\circ\phi_B}\arrow[u,Rightarrow]{r}[swap]{\sim}
        &V(S)\circ V(\textup{I}_B)
        \arrow[u,Rightarrow]{}[swap]{v(S,\textup{I}_B)}
      \end{tikzcd}
    \]
    \caption{\small Coherence of the right identity (and a similar
      diagram is drawn for the the left identity)}
    \label{fig:right-identity-coherence}
  \end{figure}
  A similar diagram for the left identity commutes.
\end{enumerate}
\end{definition}

\section{The bicategory of topological correspondences}

\subsection{Isomorphism of topoogical correspondences}
\label{sec:isom-topo-corr}

Let \(X\) be a space, and let \(\lambda\) and \(\lambda'\) be
equivalent Radon measures on it. Thus the Radon-Nikodym derivatives
\(\dd\lambda/\dd\lambda'\) and \(\dd\lambda'/\dd\lambda\) are,
respectively, \(\lambda\) and \(\lambda'\)-almost everywhere
positive. Moreover, the equality
\(\dd\lambda/\dd\lambda' \cdot \dd\lambda'/\dd\lambda=1\) holds
\(\lambda\) or\ \(\lambda'\)\nb-almost everywhere. Assume that \(Y\)
is another space and \(\pi\colon X\to Y\) is a homeomorphism. Then the
measure \(\lambda\colon \Contc(X)\to \C\) induces the measure
\(\pi_*(\lambda)\colon \Contc(Y)\to \C\) on \(Y\) as follows: for
\(f\in \Contc(Y)\), \(\pi_*(\lambda)(f)=\lambda(f\circ \pi)\). We call
\(\pi_*(\lambda)\) the push-forward (measure) of \(\lambda\).

\begin{definition}\label{def:equivalent-measure-families}
  Let \(\pi\colon X\to Y\) be an open surjection, and \(\lambda\) and
  \(\lambda'\) families of measures along \(\pi\). We call \(\lambda\)
  and \(\lambda'\) equivalent if,
  \begin{enumerate}[(i),leftmargin=*]
  \item for each \(y\in Y\), \(\lambda_y \) and \(\lambda'_y\) are
    equivalent,
  \item the function \({\dd\lambda}/{\dd\lambda'}\colon X\to \R\)
    given by
    \(\dd\lambda/\dd\lambda'(x)\mapsto
    \frac{\dd\lambda_{\pi(x)}}{\dd\lambda'_{\pi(x)}}(x)\) is
    continuous.
  \end{enumerate}
  In this case, we write \(\lambda\sim\lambda'\); we call the function
  \(\dd\lambda/\dd\lambda'\) the Radon-Nikodym derivative of
  \(\lambda\) with respect to \(\lambda'\). Note that in~{(ii)} above,
  \(\dd\lambda_{\pi(x)}/\dd\lambda'_{\pi(x)}\) is the Radon-Nikodym
  derivative of \(\lambda_{\pi(x)}\) with respect to
  \(\lambda'_{\pi(x)}\).
\end{definition}

Note that if, in Definition~\ref{def:equivalent-measure-families},
\(Y\) is singleton, then two measures \(\lambda\sim\lambda'\) on \(X\)
if the Radon-Nikodym derivative \({\dd\lambda}/{\dd\lambda'}\) is
continuous.  Also notice that in
Definition~\ref{def:equivalent-measure-families}, \emph{continuously
  equivalent} families of measures is a better terminology than
\emph{equivalent}. However, in our work, we shall not encounter any
instance of families of measures which are not continuously
equivalent. Therefore, we choose the present terminology. In~{(ii)} of
the the same definition, we are asking the not only the Radon-Nikodym
derivative \(\frac{\dd\lambda_{y}}{\dd\lambda'_{y}}\) is continuous
for \(y\in Y\) but also that the family of functions
\(\{{\dd\lambda_{\pi(x)}}/{\dd\lambda'_{\pi(x)}}\}_{x\in X}\) is
continuous \emph{in the direction} of \(X\); this transverse
continuity is used to prove
Proposition~\ref{prop:vertical-functoriality}.

Finally, in Definition~\ref{def:equivalent-measure-families}, for each
\(y\in Y\), the continuity of the Radon-Nikodym derivative
\(\dd\lambda_y/\dd\lambda'_{ y}\) implies that the function
\(\dd\lambda/\dd\lambda'>0\).  \medskip

Let \(X,Y\) and \(Z\) be spaces,
\(X\xrightarrow{\pi_X}Z\xleftarrow{\pi_Y}Y\) maps, and let \(\lambda\)
be a family of measures along \(\pi_X\). Let \(f\colon X\to Y\) be a
homeomorphism such that \(\pi_X=\pi_Y\circ f\). Then
\(\{f_*(\lambda_z)\}_{z\in Z}\) is continuous family of measures which
we denote by \(f_*(\lambda)\), thus we can write \(f_*(\lambda)_z\)
for \(f_*(\lambda_z)\).
\begin{lemma}
  \label{lem:equi-measures-image}
  Let \(X_1,X_2\) and \(Z\) be spaces, let \(\pi_i\colon X_i\to Z\) be
  maps for \(i=1,2\). Let \(a\colon X_1\to X_2\) be a homeomorphism
  such that \(\pi_1=\pi_2\circ a\). Assume that \(\lambda\) and
  \(\mu\) are equivalent family of measures along \(\pi_1\) with
  Radon-Nikodym derivative \(\dd\lambda/\dd\mu\). Then
  \(a_*(\lambda)\sim a_*(\mu)\) and the Radon-Nikodym derivative
  \(\dd a_*(\lambda)/\dd a_*(\mu)=\dd \lambda/\dd \mu\circ
  a\inverse\).
\end{lemma}
\begin{proof}
  Follows from a direct computation.
\end{proof}

\begin{lemma}
  \label{lem:equi-functoriality}
  Let \(i, X_i,Z_i,\pi_i\) and \(a_i\) be as in the succeeding lemma,
  Lemma~\ref{lem:chain-rule}. If~\(\lambda_1\) is a family of measures
  along \(\pi_1\), then
  \({(a_2\circ a_1)}_*(\lambda_1)={a_2}_*({a_1}_*(\lambda_1))\).
\end{lemma}
\begin{proof}
  This follows directly from the definition of push-forward of a
  measure.
\end{proof}
\begin{lemma}[Chain rule]
  \label{lem:chain-rule}
  Let \(X_1,X_2,X_3\) and \(Z\) be space. For \(i=1,2,3\), let
  \(\pi_i:X_i\rightarrow Z\) be maps and \(\lambda_i\) families of
  measures along \(\pi_i\). For \(i=1,2\), let
  \(a_i:X_i\rightarrow X_{i+1}\) be homeomorphisms such that
  \(\pi_i=\pi_{i+1}\circ a_i\). 
  If \({a_i}_*(\lambda_i)\) is equivalent to \(\lambda_{i+1}\) for
  \(i=1,2\), then \({(a_2\circ a_1)}_*(\lambda_1)\) is equivalent to
  \(\lambda_3\). Moreover, the Radon-Nikodym derivative satisfy the
  following relation
  \[
    \frac{\dd{(a_2\circ a_1)}_*(\lambda_1)}{\dd{\lambda_3}} =
    \frac{\dd {a_1}_*(\lambda_1)}{\dd \lambda_2}\circ a_2\inverse
    \cdot \frac{\dd {a_2}_*(\lambda_2)}{\dd{\lambda_3}}.
  \]
\end{lemma}
\begin{proof}
  This is a straightforward computation: For \(z\in Z\) and
  \(f\in \Contc(X_3)\),
  \begin{align*}
    &\int_{X_3} f(x) \frac{\dd{a_1}_*({\lambda_1}_z) }{\dd{\lambda_2}_z}\circ a_2\inverse(x) \frac{\dd {a_2}_*({\lambda_2}_z)}{\dd{\lambda_3}_z}(x)\, \dd{\lambda_3}_z(x)\\
    &=\int_{X_2} f\circ a_2 (y)\, \frac{\dd{a_1}_*({\lambda_1}_z)}{\dd{\lambda_2}_z}(y) \; \dd{\lambda_2}_z(y)
      = \int_{X_1} f\circ a_2\circ a_1(w) \,\dd{\lambda_1}_z(w)\\
    &=\int_{X_3} f(x)\,\dd{(a_2\circ a_1)}_*(\lambda_1)(x). \qedhere
  \end{align*}
\end{proof}

\begin{corollary}[Of the chain rule]
  \label{cor:iso-reflexivity}
  Let \(X_i,\pi_i,\lambda_i\), for \(i=1,2\),\(Z\) and \(a_1\) be as
  in Lemma~\ref{lem:chain-rule} above. If
  \({a_1}_*(\lambda_1)\sim \lambda_2\), then
  \(\lambda_1\sim{a_1}_*\inverse(\lambda_2)\).
\end{corollary}
\begin{proof}
  Apply the chain rule (Lemma~\ref{lem:chain-rule}) to
  \(X_1\xrightarrow{a_1} X_2 \xrightarrow{a_1\inverse} X_1\) to get
  \[
    1=\frac{\dd{a_1}_*(\lambda_1)}{\dd\lambda_2}\circ a_1
    \cdot\frac{\dd{a_1}\inverse_*(\lambda_2)}{\dd\lambda_1}.
  \]
  Since \(\frac{\dd{a_1}_*(\lambda_1)}{\dd\lambda_2}\) is positive
  continuous and \(a_1\) is a homeomorphism,
  \(\frac{\dd{a_1}_*(\lambda_1)}{\dd\lambda_2}\circ a_1\) is also
  positive continuous function. Thus
  \({\dd{a_1}\inverse_*(\lambda_2)}/{\dd\lambda_1}>0\) and continuous,
  that is, \(\dd{a_1}\inverse_*(\lambda_2)\sim {\dd\lambda_1}\).
\end{proof}

Following is a lemma that will prove useful in many computations
later.
\begin{lemma}
  \label{lem:equi-measures-on-quotient}
  Let \((G,\alpha)\) be a groupoid equipped with a Haar system. Let
  \(m\) and \(m'\) be \(G\)\nb-invariant equivalent measures on
  \(\base[G]\) with continuous Radon-Nikodym derivative
  \(\dd m/\dd m'\). Then the following hold.
  \begin{enumerate}[(i)]
  \item On \(G\), \(m\circ \alpha\sim m'\circ \alpha\) and
    \(m\circ \alpha\inverse\sim m'\circ \alpha\inverse\). Moreover,
    the Radon-Nikodym derivatives
    \[
      \frac{\dd m\circ \alpha\inverse}{\dd m'\circ \alpha\inverse} =
      \frac{\dd m}{\dd m'} \circ s_G \quad\text{and}\quad \frac{\dd
        m\circ\alpha}{\dd m'\circ \alpha} = \frac{\dd m}{\dd m'} \circ
      r_G.
    \]
  \item the Radon-Nikodym derivative \(\dd m/\dd m'\) is
    \(G\)\nb-invariant.
  \end{enumerate}
  Additionally, assume that \(G\) is proper, \(\base[G]/G\)
  paracompact. Let \(q\colon \base[G]\to \base[G]/G\) be the quotient
  map. Let \(\mu\) and \(\mu'\) be the measures on \(\base[G]/G\)
  which give the disintegration \( \mu\circ [\alpha]=m\) and
  \(\mu'\circ[\alpha]=m'\). Let \([\dd m/\dd m']\) denote the function
  which \(\dd m/\dd m'\) induce on \(\base[G]/G\) (cf.\,(ii)
  above). Then
  \begin{enumerate}[(i), resume]
  \item \(\mu\sim\mu'\) and the Radon-Nikodym derivative
    \(\dd \mu/\dd \mu' =[\dd m/\dd m]\).
  \end{enumerate}
\end{lemma}
We know that every groupoid \(G\) acts on it space of orbits (from
right) by \(u\cdot \gamma=s_G(\gamma)\) for \(u\in\base[G]\) and
\(\gamma\in G^u\). For a proper groupoid, this action is proper. A
function \(\phi\colon\base[G]\to \C\) is called invariant (under this
action) if \(\phi(u \cdot \gamma)=\phi(u)\), that is,
\(\phi(r_G(\gamma))=\phi(u)\). This is equivalent to saying that
\(\phi\circ r_G=\phi\circ s_G\) on \(G\).
\begin{proof}[Proof of Lemma~\ref{lem:equi-measures-on-quotient}]
  (i): We shall check that
  \(m\circ \alpha\inverse\sim m'\circ \alpha\inverse\). Let
  \(f\in\Contc(G)\). Then
  \begin{align*}
    \int_G f(\gamma)\,\dd m\circ\alpha\inverse(\gamma)
    &\defeq
      \int_{\base}\int_G f(\gamma\inverse)\,\dd \alpha^x(\gamma)\,\dd
      m(x)\\
    &=\int_{\base}\int_G f(\gamma\inverse)\,\dd
      \alpha^x(\gamma)\frac{\dd m}{\dd m'}(x)\,\dd
      m'(x).
  \end{align*}
  Since \(x=r_G(\gamma)\), we may write the last term above as
  \[
    \int_{\base}\int_G f(\gamma\inverse) \frac{\dd m}{\dd
      m'}(r_G(\gamma))\,\dd \alpha^x(\gamma)\,\dd m'(x)
  \]
  which, in turn, equals
  \[
    \int_{\base}\int_G f(\gamma\inverse) \frac{\dd m}{\dd
      m'}(s_G(\gamma\inverse))\,\dd \alpha^x(\gamma)\,\dd m'(x)\defeq
    \int_G f(\gamma)\frac{\dd m}{\dd m'}\circ s_G(\gamma)\,\dd m'\circ
    \alpha\inverse(\gamma).
  \]
  Thus \(m\circ \alpha\inverse\sim m'\circ \alpha\inverse\) and the
  Radon-Nikodym derivative
  \(\dd m\circ \alpha\inverse/ \dd m'\circ \alpha\inverse = \dd m/\dd
  m' \circ s_G\). The other claim can be proved along similar lines.
  
  \noindent (ii): Since \(m\) (or \(m'\)) is an invariant measure on
  \(\base[G]\), we have \(m\circ \alpha=m\circ \alpha\inverse\) (and
  similar for \(m'\)). Which along with (i) above says that
  \begin{equation*}
    \frac{\dd m}{\dd m'} \circ s_G = \frac{\dd m\circ
      \alpha\inverse}{\dd m'\circ \alpha\inverse} = \frac{\dd
      m\circ\alpha}{\dd m'\circ \alpha} = \frac{\dd m}{\dd m'} \circ
    r_G.
  \end{equation*}
  In other words, \(\dd m/\dd m'\) is an invariant function on
  \(\base[G]\).
   
  \noindent (iii): First of all, we note that functions
  \(\frac{\dd m}{\dd m'}\) and \([\frac{\dd m}{\dd m'}]\) have the
  same images in \(\R\). Now, given \(m\) (or \(m'\)), recall the
  definition of \(\mu\) (or \(\mu'\), respectively) from
  Proposition~{3.1}(ii) in~\cite{Holkar2017Composition-of-Corr}. Let
  \(f\in \Contc(\base[G]/G)\) and
  \(e\colon \base[G]\to \R^*\cup\{0\}\) be a cutoff function for the
  quotient map \(q\). Then
  \begin{equation*}
    \mu(f)\defeq m((f\circ q) \cdot e)=m'\left(f\circ q\cdot \frac{\dd m}{\dd m'}
      \cdot e\right)= m'\left(f \cdot \left[\frac{\dd m}{\dd m'}\right] \cdot
      e\right)\defeq m'\left(f\cdot \left[\frac{\dd m}{\dd m'}\right]\right).
  \end{equation*}
  Thus \(\mu\sim \mu'\) and
  \(\frac{\dd \mu}{\dd \mu'}=\left[\frac{\dd m}{\dd m'}\right]\).
\end{proof}

\begin{definition}[Isomorphism of topological correspondences]
  \label{def:iso-corr}
  Let \((X, \lambda)\) and \((X', \lambda')\) be topological
  correspondences from \((G, \alpha)\) to \((H, \beta)\). An
  isomorphism \((X, \lambda, \Delta)\to (X', \lambda', \Delta')\) is a
  \(G\)-\(H\)\nb-equivariant homeomorphism \(\phi: X \rightarrow X'\)
  with \(\phi_*(\lambda)\sim\lambda'\).
\end{definition}


Following is an example of isomorphism of correspondences;
Proposition~\ref{prop:lifted-cocycles-are-iso} gives a class of
isomorphism correspondences.
\begin{example}[The identity isomorphism of identity correspondence]
  \label{exa:id-corr} In this example, we discuss the left and right
  identity isomorphisms.  Let \((G,\alpha)\) and \((H,\beta)\) be
  topological groupoids with Haar systems. Recall from
  Example~\ref{exa:id-corr-iso} that \((G,\alpha\inverse)\) is a
  topological correspondence on \((G,\alpha)\), and similar is for
  \((H,\beta)\). Let \((X,\lambda)\colon (G,\alpha)\to (H,\beta)\) be
  a topological correspondence. What are composites of
  \((G,\alpha\inverse)\) and \((X,\lambda)\), and \((X,\lambda)\) and
  \((H,\beta\inverse)\)?

  \noindent \textit{The left identity isomorphism:} The quotient
  \((G\times_{\base[G]}X)/G\) is isomorphic to \(X\); the map
  \(i\colon (G\times_{\base[G]}X)/G \to X\) given by
  \(i([\gamma\inverse,x])=\gamma\inverse x\), where
  \([\gamma\inverse, x]\in (G\times_{\base[G]}X)/G\) is the
  equivalence class of \((\gamma\inverse,x)\in G\times_{\base[G]}X\),
  induces this homeomorphism. The inverse of \(i\) is given by
  \(i\inverse(x)=[r_X(x),x]\). Moreover, the map~\(i\) is
  \(G\)-\(H\)\nb-equivariant. Thus we identify \(X\)~as the quotient
  space \((G\times_{\base[G]}X)/G\) and call the map
  \[
    q\colon G\times_{\base[G]}X \to X, \quad q\colon (\gamma\inverse,
    x)\mapsto \gamma\inverse x
  \]
  the quotient map\footnote{At this point, we avoid using a different
    notation --- say \(X'\) --- for the quotient space
    \((G\times_{\base[G]}X)/G\) as it will increase complexity in
    writing.}. The family of measures \([\alpha]\) ---as in~{(2)} on
  page~\pageref{page:compo-rev-2}--- along the quotient map \(q\) is
  given by
  \[
    \int_{G\times_{\base[G]}X} f(t)\,\dd[\alpha]_{x}(t)=\int_G
    f(\gamma,\gamma\inverse x)\,\dd\alpha^{r_X(x)}(\gamma)
  \]
  where \(f\in \Contc(G\times_{\base[G]}X)\) and \(x\in X\).

  Let \(T\) denote the transformation groupoid
  \((G\times_{\base[G]}X)\rtimes G\) for the diagonal action of \(G\)
  on \(G\times_{\base[G]}X\). Then we know that \(T\) is a proper
  groupoid with \(\base[T]=G\times_{\base[G]}X\). Moreover, the Haar
  system \(\alpha\) on \(G\) induces a Haar system \(\tilde{\alpha}\)
  on \(T\). Fix \(u\in\base[H]\). Define the measure
  \(m_u=\alpha\times\lambda_u\) on \(\base[T]\) as in ---as in~{(1)}
  on page~\pageref{page:compo-rev-1}. Then \(m_u\) is
  \((T,\tilde{\alpha})\)\nb-quasi-invariant measure on \(\base[T]\)
  (cf.\,(2) on page~\pageref{page:compo-rev-3}). The modular function
  \(\Delta_1\) of \(T\) associated with \(m_u\) is
  \begin{equation}\label{eq:left-id-corr-1}
    \Delta_1(\gamma\inverse, x, \eta)=\Delta(\eta\inverse,x) \qquad
    \text{
      (see~Equation~\eqref{equ:composition-1-cocycle})}
  \end{equation}

  where \(\Delta\) is the adjoining function of \((X,\lambda)\). Now
  we get a 0\nb-cocycle \(b\colon \base[T]\to\R^+\) on \(T\) such that
  \(d^0(b)=\Delta_1\) and \(bm_u\) is a
  \((T,\tilde{\alpha})\)\nb-invariant measure on \(\base[T]\). This
  measure produces a measures \(\mu_u\) on \(X\homeo \base[T]/T\) such
  that \((X,\{\mu_u\}_{u\in\base[H]})\) composite of \((G,\alpha)\)
  and \((X,\lambda)\). Write \(\{\mu_u\}_{u\in\base[H]}=\mu\). In what
  follows, we shall show that the identity map on \(X\) produces an
  isomorphism of topological correspondences \((X,\mu)\) and
  \((X,\lambda)\).

  Fix \(u\in\base[H]\). The measure \(\lambda_u\circ [\alpha]\) on
  \(\base[T]\) is also (\((T,\tilde{\alpha})\)-)invariant due to
  Proposition~{3.1}(i)
  in~\cite{Holkar2017Composition-of-Corr}. Moreover, the third part of
  the same Proposition shows that \(\lambda_u\) is the unique measure
  on \(X\) that disintegrates \(\lambda_u\circ[\alpha]\) along
  \([\alpha]\).  What is the relation between
  \(\lambda_u\circ [\alpha]\) and \(bm_u\)? We claim that they are
  equivalent invariant measures on \(\base[T]\) with a continuous
  Radon-Nikodym derivative. If the claim holds, then
  Lemma~\ref{lem:equi-measures-on-quotient}(iii) applied to
  \((T,\tilde{\alpha}), bm_u\) and \(\lambda_u\) will show that
  \(\mu_u\sim \lambda_u\) and the Radon-Nikodym derivative
  \(\dd\mu_u/\dd\lambda_u\) is continuous. This proves that the
  identity map of \(X\) induces the isomorphism of topological
  correspondences \((X,\mu)\) and \((X,\lambda)\).

  Now we prove the claim, namely, \(bm_u\) and
  \(\lambda_u\circ [\alpha]\) are equivalent invariant measures on
  \(\base[T]\) having continuous Radon-Nikodym derivative. 
  The first observation is that \(\lambda_u\circ [\alpha]\) and
  \(m_u\) are equivalent. This follows because for any
  \(f\in \Contc(\base[T])\),
  \begin{align*}
    m_u(f)&=\int_X\int_G f(\gamma\inverse,x)
            \,\dd\alpha^{u}(\gamma)\dd\lambda_u(x)\\
          &=\int_X\int_G f(\gamma,\gamma\inverse
            x) \Delta(\gamma,\gamma\inverse
            x)\,\dd\alpha^{u}(\gamma)\dd\lambda_u(x)
            = \lambda_u\circ [\alpha](f\cdot\Delta).
  \end{align*}
  The second equality above is the definition of adjoining function
  \(\Delta\) (see Definition~\ref{def:correspondence}(iv)) of
  \((X,\lambda)\). Thus \(m_u\sim\lambda_u\circ[\alpha]\) with the
  Radon-Nikodym derivative
  \(\dd m_u/\dd \lambda_u\circ [\alpha]= \Delta>0\). Since \(b\) is
  0-cocycle in the \(\R^+\)\nb-valued cohomology of \(T\) (with
  \(d^0(b)=\Delta_1>0\)), we get that \(bm_u\sim m_u\). Therefore,
  \(bm_u\sim \lambda_u\circ [\alpha]\) and the Radon--Nikodym
  derivative \(\dd bm_u/ \dd \lambda_u\circ [\alpha]=b
  \Delta\). Moreover, Lemma~\ref{lem:equi-measures-on-quotient}{(iii)}
  says that \(\mu_u\sim \lambda_u\) with
  \(\dd\mu_u/\dd \lambda_u= [\dd bm_u/ \dd \lambda_u\circ [\alpha]=[b
  \Delta]\); note that both \(b\) and \(\Delta\) are
  \(H\)\nb-invariant, therefore, \([b\Delta]\) makes sense.

  Before we finish the discussion, we simplify the function
  \(b\Delta\) so that it will prove useful in later computations. For
  \((\eta\inverse,x)\in \base[T]\),
  \begin{equation*}
    b\Delta(\eta\inverse,x)=b(\eta\inverse,x)\Delta(\eta\inverse,x).
  \end{equation*}
  Using Equation~\eqref{eq:left-id-corr-1} and the fact that
  \(\Delta=b\circ s_G/b\circ r_G=d^0(b)\), we can see that last term
  above equals
  \begin{multline*}
    b(\eta\inverse,x)\Delta_1(\eta\inverse,x,\eta)=b(\eta\inverse,x)\frac{b\circ
      s_T(\eta\inverse,x,\eta)}{b\circ
      r_T(\eta\inverse,x,\eta)}\\
    =b(\eta\inverse, x)\frac{b(s_H(\eta), \eta\inverse
      x)}{b(\eta\inverse,x)} = b(s_H(\eta), \eta\inverse x).
  \end{multline*}
  We know that the space of units of the transformation groupoid
  \(H\ltimes X\) is identified with \(X\) via the homeomorphism
  \(\xi \colon X\to H\ltimes X \to X, \xi\colon x\mapsto
  (r_X(x),x)\). Now note that \(b\) is already a function on
  \(\base[T]\), therefore, the Radon--Nikodym derivative
  \begin{equation}
    \label{eq:left-id-corr-2}
    \frac{\dd\mu_u}{\dd
      \lambda_u}(\eta\inverse,x)=[b\Delta(\eta\inverse,x)]=
    b(s_H(\eta),\eta\inverse x)=b\circ s_{H\ltimes X}(\eta\inverse, x)=\xi(\eta\inverse x)
  \end{equation}
  for all \((\eta\inverse, x)\in H\ltimes X\).  \smallskip

  \noindent \emph{The right identity isomorphism:} In a similar
  fashion as for the left identity isomorphism, one may prove that any
  composite
  \((X,\lambda)\circ (H,\beta\inverse)\colon (G,\alpha)\to (H,\beta)\)
  is isomorphic to \((X,\lambda)\). The equivariant homeomorphism of
  spaces \((X\times_{\base[H]}H)/H\homeo X\) is clear. While
  constructing the measures, note that the adjoining function of
  \((H,\beta)\) is the constant function 1,
  see~Example~\ref{exa:id-corr-iso}. Therefore, the function
  \(\Delta_1\) above is also constant~1. Choose any 0\nb-cocycle \(b\)
  on the transformation groupoid
  \(T\defeq (X\times_{\base[H]}H)\rtimes H\) with
  \(d^0(b)=\Delta_1=1\) and \(b m_u\) is an invariant measure (where
  \(m_u\defeq \lambda\times\beta\inverse_u\)) on the space of units of
  \(T\). Then
  \[
    1=d^0(b)\defeq \frac{b\circ s_T}{b\circ r_T}, \quad \text{ that
      is, } \quad b\circ r_T=b\circ s_T.
  \]
  Thus \(b\) induces a function \([b]\) on
  \(X\homeo (X\times_{\base[H]}H)/H\). With this observation, the
  discussion for left identity correspondence holds here
  word-to-word. And we get that any composite \((X\circ H,\mu)\) of
  \((X,\lambda)\) and \((H,\beta\inverse)\) is isomorphic to
  \((X,\lambda)\) with the homeomorphism given above and the
  Radon-Nikodym derivative \(\dd\mu_u /\dd\lambda_u=[b]\).
\end{example}

\smallskip

Recall from Definition~\ref{def:composition} that a composite of two
topological correspondences is not defined uniquely; it depends on a
0\nb-cocycle on a certain transformation groupoid. Now we show that
any two composites of topological correspondences are isomorphic in
the sense of Definition~\ref{def:iso-corr}.
\begin{proposition}
  \label{prop:lifted-cocycles-are-iso}
  Let
  \begin{align*}
    (X, \alpha)&\colon (G_1, \lambda_1) \rightarrow (G_2, \lambda_2)\\
    (Y, \beta)&\colon (G_2, \lambda_2) \rightarrow (G_3, \lambda_3)
  \end{align*}
  be topological correspondences, and let
  \[
    (\Omega, \mu), (\Omega, \mu')\colon (G_1,\lambda_1)\rightarrow
    (G_3, \lambda_3)
  \]
  be two composites of these correspondences. Assume that
  \((\Omega, \mu)\) is obtained by using a \(0\)\nb-cochain \(b\) and
  \((\Omega',\mu')\) is obtained by using a 0\nb-cochain \(b'\) on the
  transformation groupoid \((X\times_{\base[G]_2}Y)\rtimes G_2\).
  Then \((\Omega, \mu)\) and \( (\Omega, \mu')\) are isomorphic
  topological correspondences.
\end{proposition}
\begin{proof}
  In this proof we use the notation which is used after
  Theorem~\ref{thm:mcorr-gives-ccorr} to state
  Definition~\ref{def:composition}. Let
  \(Z\defeq X\times_{\base[G]_2} Y\) and
  \(\pi\colon Z\rightarrow\Omega\defeq Z/G_2\) be the quotient
  map. Since \(b,b'\in\CC^0_{G_3}(Z,\R_+^*)\) with
  \(d^0(b) = d^0(b')=\Delta\), Remark~{1.14}
  in~\cite{Holkar2017Construction-of-Corr} gives a positive function
  \(c\colon\Omega\to\R^*_+\) with the property that
  \begin{equation}
    \label{eq:-1-cocycle-c}
    b' = (c\circ \pi ) b.
  \end{equation}
  Since the quotient map \(\pi\) is open, the continuity of \(b\),
  \(b'\) implies that any function \(c\) satisfying
  Equation~\eqref{eq:-1-cocycle-c} is continuous. Let
  \(f\in \Contc(\Omega)\) and \(u\in \base[G]_3\). Then using
  Equation~\eqref{eq:composite-measure} we write
  \begin{align*}
    \int_\Omega f[x,y]\,\dd\mu_u'([x,y])
    &=\int_Y\int_X f\circ\pi(x,y) e(x,y)\, b'(x,y)\, \dd\alpha_{r_Y(y)}(x)\, \dd\beta_u(y)\\ 
    &=\int_Y\int_X f\circ\pi(x,y) e(x,y)\, c\circ\pi(x,y) b(x,y)\, \dd\alpha_{r_Y(y)}(x)\, \dd\beta_u(y) \\
    &=\int_\Omega f[x,y] c[x,y]\,\dd\mu_u([x,y])
  \end{align*}
  where \(e\colon Z\to \R^+\) is a cutoff function. This calculation
  shows that for every \(u\in \base[G]_3\), \(\mu_u'\sim\mu_u\) with
  the Radon-Nikodym derivative \(\frac{\dd\mu'_u}{\dd\mu_u}=c\), where
  \(c\colon \Omega\to\R_+^*\) is a function satisfying
  Equation~\eqref{eq:-1-cocycle-c}.
\end{proof}
\medskip

Let \((G,\alpha)\) and \((H,\beta)\) be locally compact groupoids
equipped with Haar systems. For \(i=1,2,3\), let \((X_i, \lambda_i)\)
be a topological correspondence from \((G,\alpha)\) to
\((H,\beta)\). Assume that, for \(i=1,2\),
\(\phi_i\colon X_i\rightarrow X_{i+1}\) is an isomorphism of
correspondences. Then the composite
\(\phi_2\circ\phi_1\colon X_1\rightarrow X_3\) gives an isomorphism of
correspondences--- this follows because of the fact that the composite
of \(G\)-\(H\)\nb-equivariant maps is also an equivariant map and the
chain rule.

Let \((G,\alpha)\) and \((H,\beta)\) be groupoids with Haar
systems. Then isomorphism is an equivalence relation on the set of
topological correspondences from \((G,\alpha)\) to \((H,\beta)\): Let
\((X,\lambda_1)\), \((Y,\lambda_2)\) and \((Z,\lambda_3)\) be
correspondences from \((G,\alpha)\) to \((H,\beta)\).
\begin{description}[font=\normalfont]\setlength{\itemsep}{0pt}
\item[Reflexivity] is given by the identity function on \(X\).
\item[Symmetry] if \(\phi\) is an isomorphism from \((X,\lambda_1)\)
  to \((Y,\lambda_2)\), then \(\phi\inverse\) is
  \(G\)-\(H\)\nb-equivariant homeomorphism. Now use
  Corollary~\ref{cor:iso-reflexivity} to see that
  \(\phi\inverse_*(\lambda_2)\sim\lambda_1\).
\item[Transitivity] Follows from the discussion just before this
  paragraph.
\end{description}

\begin{remark}
  \label{rem:for-hor-2-comp}
  Let \((G,\alpha),(H,\beta)\) and \((K,\kappa)\) be groupoids with
  Haar systems. Let \((X,\lambda)\) and \((Y,\kappa)\) be
  correspondences from \((G,\alpha)\) to \((H,\beta)\), and
  \((X',\lambda')\) and \((Y',\kappa')\) be correspondences from
  \((H,\beta)\) to \((K,\mu)\). Let
  \((X\circ X',\lambda\circ\lambda')\) and
  \((Y\circ Y',\kappa\circ\kappa')\) be some composites of
  \((X,\lambda)\) and \((X',\lambda')\), and \((Y,\kappa)\) and
  \((Y',\kappa')\), respectively, see Figure~\ref{fig:vert-comp}. Thus
  \(X\circ X'= (X\times_{\base[H]}X')/H\), and the family of measures
  \(\lambda\circ\lambda'\) is given by~{(ii)}
  in~Definition~\ref{def:composition}. Similar is for \(Y\circ Y'\)
  and \(\kappa\circ\kappa'\). Let \(b_1\) and \(b_2\) be the cochains
  in appropriate groupoid cohomologies used to produce
  \(\lambda\circ\lambda'\) and \(\kappa\circ\kappa'\), respectively.

  Additionally, assume that \(\phi\colon X\to Y\) and
  \(\phi'\colon X'\to Y'\) are isomorphisms of correspondences. Then
  the map
  \([\phi\times\phi']\colon (X\times_{\base}X')/H\to
  (Y\times_{\base}Y')/H\) is an isomorphism of correspondences where
  \([\phi\times\phi']([x,y])=[\phi(x),\phi'(y')]\), see
  Figure~\ref{fig:vert-comp}.
  \begin{figure}[htb]
    \(\xymatrix@1@C=3em{ (G,\alpha)\quad
      \ar@/^1pc/[r]^{(X,\lambda)}_{}="0"
      \ar@/_1pc/[r]_{(Y,\kappa)}^{}="1" \ar@{=>}"0";"1"^{\phi} &
      \quad(H,\beta)\quad \ar@/^1pc/[r]^{(X',\lambda')}_{}="0"
      \ar@/_1pc/[r]_{(Y',\kappa')}^{}="1" \ar@{=>}"0";"1"^{\phi'}&
      \quad(K,\mu) }\) \qquad
    \(\xymatrix@1@C=5em{ (G,\alpha)\quad \ar@/^1pc/[r]^{(X\circ
        X',\lambda\circ \lambda')}_{}="0" \ar@/_1pc/[r]_{(Y\circ
        Y',\kappa\circ \kappa')}^{}="1" \ar@{=>}"0";"1"^{[\phi\times
        \phi']} & \quad(K,\mu) }\)
    \caption{}
    \label{fig:vert-comp}
  \end{figure}

  To check this, first of all, note that
  \(\phi\times \phi'\colon X\times_{\base[H]}X'\to
  Y\times_{\base[H]}Y'\) and
  \([\phi\times \phi']\colon X\circ X'\to Y\circ Y'\) are well-defined
  \(G\)-\(K\)\nb-equivariant homeomorphisms. Then, since
  \(\phi_*(\lambda)\sim \kappa\) and
  \({\phi'}_*(\lambda')\sim \kappa'\), we get
  \[
    (\phi\times\phi')_*(\lambda \times \lambda') \sim \kappa \times
    \kappa';
  \] the Radon-Nikodym derivative
  \[
    \frac{ \dd (\phi\times\phi')_*(\lambda \times \lambda')} { \dd
      \left(\kappa \times \kappa' \right)} = \frac{\dd
      \phi_*(\lambda)} {\dd\kappa}\; \frac{ \dd\phi_*'(\lambda')} {\dd
      \kappa'}.
  \]
  This discussion along with the fact that \(b_1\) and \(b_2\) are
  continuous positive functions allows us to say that
  \begin{equation}
    \label{eq:vert-comp}
    b_1\cdot (\phi\times\phi')_*(\lambda \times \lambda') \sim b_2\cdot(\kappa \times \kappa');
  \end{equation}
  the Radon-Nikodym derivative
  \[
    \frac{ \dd(b_1\, (\phi\times\phi')_*(\lambda \times \lambda'))} {
      \dd \left(b_2\,\kappa \times \kappa' \right)_u} =
    \frac{\dd\phi_*(\lambda)} {\dd\kappa}\; \frac{ \dd
      \phi_*'(\lambda'_u)}{\dd \kappa'}\;\frac{b_1}{b_2}.
  \]
  Now use Lemma~\ref{lem:equi-measures-on-quotient}{(iii)} on the
  transformation groupoid \((Y\times_{\base}Y') \rtimes H\) to
  conclude that \(\lambda\circ\lambda'\sim\kappa\circ\kappa'\). This
  remarks shows in the bicategory of topological correspondences, two
  2\nb-arrows between a two composable 1\nb-arrows induce a 2\nb-arrow
  between a composite of 2\nb-arrows.
\end{remark}

\subsection{The bicategory of topological correspondences}
\label{sec:bicat-topol-corr}

Now, with the help of the discussion until now in our hand, we can
define the data to form the bicategory of topological correspondences:
\begin{description}[font=\normalfont\itshape,leftmargin=*]\setlength{\itemsep}{0pt}
\item[Objects or vertices:] second countable, locally compact,
  Hausdorff groupoids with Haar systems.
\item[\(1\)\nb-arrows or edges] topological correspondences in which
  the space is locally compact, Hausdorff, second countable.
\item[\(2\)\nb-arrows or \(2\)\nb-cells] isomorphisms of topological
  correspondences (Definition~\ref{def:iso-corr}).
\item[Vertical composition of 2-arrows] 2\nb-arrows are merely
  functions between spaces; their composition is the usual composition
  of functions.
\item[\(1\)\nb-identity arrow] the identity \(1\)\nb-arrow on
  \((G, \alpha)\) is \((G, \alpha\inverse)\), see
  Example~\ref{exa:id-corr-iso}.
\item[\(2\)\nb-identity arrow] the identity \(2\)\nb-arrow on a
  topological correspondence \((X,\lambda)\) is the identity map
  \(\textup{Id}_X\colon X\rightarrow X\).
\item[Composition of 1-arrows] composition of correspondences as in
  Definition~\ref{def:composition}.
\item[Horizontal composition of 2-arrows] with the data in
  Remark~\ref{rem:for-hor-2-comp}, we call \([\phi\times\phi']\) the
  horizontal product of \(\phi\) and \(\phi'\).
\item[The associativity isomorphism] described in
  Theorem~\ref{thm:proof-of-bicategory-of-top-corr} below.
\item[The identity isomorphism] described in Example~\ref{exa:id-corr}
  earlier.
\end{description}
Thus now we have the data required in \(i\)--\(iv\) in
Definition~\ref{def:bicat}. The following theorem describes how this
data fulfills the necessary conditions to give us the bicategory of
topological correspondences. The proof of associativity isomorphisms
is very long. Therefore, we break it into pieces and describe after
the next theorem.

\begin{theorem}
  \label{thm:proof-of-bicategory-of-top-corr}
  The above data along with (obvious) associativity and identity
  isomorphisms form the bicategory \(\tcorr\) of topological
  correspondences.\end{theorem}
\begin{proof}
  In the following discussion, the number of each topic indicates what
  topic in Definition~\ref{def:bicat} it is.

  \noindent v){\itshape Associativity
    isomorphism}: \label{page:asso-data} Firstly, we list our data and
  notation for the associativity isomorphism.
  \begin{enumerate}[i),leftmargin=*]
  \item For \(i\in\{1,2,3\}\), \((X_i,\lambda_i)\) is a correspondence
    from \((G_i,\alpha_i)\) to \((G_{i+1},\alpha_{i+1})\) with
    \(\Delta_i\) as the adjoining function;
  \item \((X_{i(i+1)},\mu_{i(i+1)})\) denotes a composite of
    \((X_i,\lambda_i)\) and \((X_{i+1},\lambda_{i+1})\) for
    \(i=1,2\). Thus \(X_{ii+1}\) denotes the quotient space
    \((X_i\times_{\base[G]_{i+1}}X_{i+1})/G_{i+1}\). We write the
    cochain in
    \(\CC^0_{G_{i+2}}((X_i\times_{\base[G]_{i+1}}X_{i+1})\rtimes
    G_{i+1},\R^*_+)\) that produces the family of measures
    \(\mu_{i(i+1)}\) as \(b_{i(i+1)}\)
    (cf.\,Definition~\ref{def:composition}).  We write \(\pi_{ii+1}\)
    for the quotient map
    \(X_i\times_{\base[G]_{i+1}}X_{i+1} \to
    (X_i\times_{\base[G]_{i+1}}X_{i+1})/G_{i+1}\).
  \end{enumerate}
  Each space \(X_i\) above is locally compact, Hausdorff, and the
  action of the groupoid \(G_{i+1}\) on it is proper for
  \(i=1,2,3\). Therefore, the diagonal action of \(G_{i+1}\) on the
  fibre product \(X_i\times_{\base[H]_{i+1}} X_{i+1}\) is
  proper. Moreover, the similar diagonal action of \(G_2\times G_3\)
  on \(X_{1}\times_{\base[G]_2}X_2\times_{\base[G]_3}X_{3}\) is
  proper.  Let \(T\) denote the proper transformation groupoid
  \((X_1\times_{\base[G]_2}X_2\times_{\base[G]_3}X_3)\rtimes
  (G_2\times G_3)\) for the diagonal action of \(G_2\times G_3\) on
  \(X_1\times_{\base[G]_2}X_2\times_{\base[G]_3}X_3\). Let
  \(\pi_{123}\) denote the quotient map \(\base[T]\to X_{123}\). The
  Haar system \(\alpha_2\times\alpha_3\) on \(G_2\times G_3\) induces
  a Haar system on \(T\) which we denote by \(\alpha_2\times\alpha_3\)
  itself. The quotient \(X_{123}=\base[T]/T\) is denoted by
  \(X_{123}\).

  Let \((X_{(12)3},\mu_{(12)3})\) be the given composite of
  \((X_{12},\mu_{12})\) and \((X_3,\lambda_3)\), and let similar be
  the meaning of \((X_{1(23)},\mu_{1(23)})\). Let
  \(\pi'_{(12)3}\colon X_{12}\times_{\base[G]_{3}}X_3\to X_{(12)3}\)
  be the quotient map, and similar be the meaning of
  \(\pi'_{1(23)}\). Let \(\pi_{(12)3}=\pi'_{(12)3}\circ
  \pi_{12}\). The proof starts now by defining two functions \(a'\)
  and \(a''\) below. All above data and these maps gives us
  Figure~\ref{fig:asso}. The map \(a\) in this figure is the
  associativity isomorphism that we explain now.

  Define
  \begin{align*}
    a'&\colon
        X_{123} \to X_{(12)3}, \text{ by }& a'\colon [x_1, x_2, x_2]
    &\mapsto [[x_1,x_2],x_3],\\
    a''&\colon
         X_{123} \rightarrow X_{1(23)}, \text{ by } &a''\colon [x_1, x_2, x_2] &\mapsto [x_1, [x_2, x_3]]
  \end{align*}
  where \([x_1,x_2,x_3]\in X_{123}\). We show that \(a'\) is
  well-defined and the well-definedness of \(a''\) can be proven along
  similar lines. Let
  \begin{align*}
    \pi_{12}\times\Id_{X_3}&\colon \base[T]\to X_{12}\times_{\base[G]_3} X_3,\\
    \pi_{(12)3}'&\colon X_{12}\times_{\base[G]_3} X_3\to X_{(12)3}
  \end{align*}
  be the quotient maps for the diagonal actions of \(H_2\) and \(H_3\)
  for the actions on \(\base[T]\) and
  \(X_{12}\times_{\base[G]_3}X_{3}\), respectively. Then
  \(\pi_{123}'\circ \pi_{12}\times\Id_{X_3}=\pi_{123}\) is a
  well-defined continuous surjection
  (cf.\,Figure~\ref{fig:asso}). Note that for
  \((x_1,x_2,x_3)\in \base[T]\) and appropriate
  \((\gamma_1,\gamma_2)\in G_1\times G_2\),
  \begin{multline*}
    \pi_{123}'(\pi_{12}\times\Id_{X_3} (x_1\gamma_1,\gamma_1\inverse
    x_2\gamma_2,\gamma_2\inverse x_3))=[[x_1\gamma_1,\gamma_1\inverse
    x_2\gamma_2],\gamma_2\inverse
    x_3]\\=[[x_1\gamma_1,\gamma_1\inverse
    x_2]\gamma_2,\gamma_2\inverse
    x_3]=[[x_1,x_2]\gamma_2,\gamma_2\inverse x_3]\\=[[x_1,x_2],
    x_3]=\pi_{123}'(\pi_{12}\times\Id_{X_3} (x_1,x_2,x_3)).
  \end{multline*}
  Therefore, due to the the universal property of the quotient space,
  \(\pi_{123}'\circ \pi_{12}\times\Id_{X_3}\) induces a continuous
  surjection \(X_{123} \to (X_1\circ X_2)\circ X_3\) which we call
  \(a'\).

  We claim that both \(a'\) and \(a''\) are homeomorphisms. We prove
  that \(a'\) is a homeomorphism, and the claim for \(a''\) can be
  proved similarly. Surjectivity of \(a'\) is built in its
  definition. To show that \(a'\) is one-to-one, assume that for some
  \([x_1,x_2,x_3],[y_1,y_2,y_3]\in X_{123}\)
  \(a'[x_1,x_2,x_3]=a'[y_1,y_2,y_3]\), that is,
  \([[x_1,x_2],x_3]=[[y_1,y_2],y_3]\). Then there is
  \(\gamma_2\in G_2\) with the property that
  \[
    ([x_1,x_2\gamma_2],\gamma_2\inverse
    x_3)=([x_1,x_2]\gamma_2,\gamma_2\inverse x_3)=([y_1,y_2],y_3).\]
  Now there is \(\gamma_1\in G_1\) such that
  \[
    (x_1\gamma_1,\gamma_1\inverse x_2\gamma_2,\gamma_2\inverse x_3)=
    (y_1,y_2,y_3).
  \]
  Thus
  \[
    [x_1,x_2,x_3]=[y_1,y_2,y_3]\in X_{123}.
  \]

  Next we show that \(a'\) is open. Let
  \(\pi_{123}\colon \base[T] \to X_{123}\) be the quotient map, and
  let \(U\subseteq X_{123}\) be open. Then \(\pi\inverse(U)\) is
  open. Since all the groupoids we are working with have open range
  maps, the quotient maps \(\pi_{12}\times\Id_{X_3}\) and
  \(\pi_{(12)3}'\) are
  open,~\cite{Muhly-Williams1995Gpd-Cohomolgy-and-Dixmier-Douady-class}*{Lemma
    2.1}. Therefore,
  \(a'(U)=p'(p(\pi\inverse(U)))\subseteq X_{(12)3}\) is open.

  Since the quotient maps in Figure~\ref{fig:asso} are
  \(G_1\)-\(G_4\)\nb-equivariant, so are \(a'\) and
  \(a''\). Eventually, we define the associativity isomorphism
  \(a(X_1,X_2,X_3)\) as
  \begin{align*}
    a(X_1,X_2,X_3)&=a^{\prime\prime }\circ
                    {a'}\inverse, \text{ that is, }\\
    a(X_1,X_2,X_3)
                  &\colon [[x_1, x_2], x_3]\mapsto [x_1, [x_2, x_3]]
  \end{align*}
  where \([[x_1, x_2], x_3]\in (X_1\circ X_2)\circ X_3\).  Whenever
  the spaces \(X_i\), in the discussion, are clear, we write \(a\)
  instead of \(a(X_1,X_2,X_3)\). We shall write \(a\) in rest of the
  discussion in this part of the proof.

  We still need to show that \(a\) satisfies~(ii) of
  Definition~\ref{def:iso-corr} to conclude that it is an isomorphism
  of correspondences. The proof of this fact is written at the end of
  the present main proof; the reader may see
  page~\pageref{page:asso-proof}. Proposition~\ref{prop:left-measures-equi}
  and Remark~\ref{rem:invariance-of-B/B} mark end of that
  proof.\medskip
  
  \noindent vi) {\itshape Identity isomorphisms}: Let \(i=1,2\) and
  \((G_i,\alpha_i)\) a groupoid with a Haar system, and let
  \((X, \lambda)\) be a correspondence from \((G_1,\alpha_1)\) to
  \((G_2,\alpha_2)\) with \(\Delta\) as the adjoining function. As we
  chose \((G_i,\alpha_i\inverse)\) is the identity arrow on
  \((G_1,\alpha)\). The claim is that the
  \(G_1\)-\(G_2\)\nb-equivariant homeomorphisms of spaces
  \begin{align*}
    l(G_1) &\colon (G_1\times_{\base[G]_1} X)/G_1 \rightarrow
             X, & [\gamma\inverse,x] &\mapsto \gamma\inverse x\\
    r(G_1) &\colon (X\times_{\base[G]_2} G_2)/G_2 \rightarrow
             X,& [x,\gamma] &\mapsto x\gamma
  \end{align*}
  are, respectively, the left and right identity coherences. This
  claim is proved, with an abuse of notation, in
  Example~\ref{exa:id-corr}. \medskip

  \noindent vii) {\itshape Horizontal composition of\; 2\nb-arrows}:
  Let \((X_i,\lambda_i), (X_i',\lambda_i')\) be correspondences from
  \((G_i,\alpha_i)\) to \((G_{i+1},\alpha_{i+1})\) for \(i=1,2\) and
  let \(\phi_i\colon X_i\to X_i'\) be isomorphisms of
  correspondences. Let \((X_{12},\mu)\) be a composite of
  \((X_1,\lambda_1)\) and \((X_2,\lambda_2)\), and \((X_{12}',\mu')\)
  a composite of \((X_1',\lambda_1')\) and \((X_2',\lambda_2')\).
  
  Now the map
  \[
    \phi_1\times \phi_2\colon X_1\times_{\base[G]_2} X_2 \to
    X_1'\times_{\base[G]_2} X_2',\quad \phi_1\times
    \phi_2(x,y)=(\phi_1(x), \phi_2(y)),
  \]
  where \((x,y)\in X_1\times_{\base[G]_2} X_2 \), is a
  \(G_1\)-\(G_3\)\nb-equivariant homeomorphism for the left and right
  obvious actions of \(G_1\) and \(G_3\). Moreover, this map is also
  \(G_2\)\nb-equivariant for the diagonal actions of \(G_2\) on the
  fibre products. Therefore, the map induces a well-defined
  \(G_1\)-\(G_3\)\nb-equivariant homeomorphism
  \( [\phi_1\times \phi_2]\colon X_{12} \to X_{12}'.  \) We define
  \([\phi_1\times \phi_2]\) the horizontal composition of the
  2\nb-arrows \(\phi_1\) and \(\phi_2\) which is often written as
  \(\phi_2\cdot_h\phi_2\). To check that this definition makes sense,
  one needs to check that \([\phi_1\times \phi_2]\) induces an
  isomorphism of topological correspondences \((X_{12},\mu)\) and
  \((X_{12}',\mu')\). This is verified in
  Remark~\ref{rem:for-hor-2-comp}.  \medskip
  
  \noindent vii) {\itshape Associativity coherence} : Let
  \((G_i, \alpha_i)\) be groupoids equipped with Haar systems for
  \(i=1,\dots,5\). Let \((X_i,\lambda_i)\) be a correspondence from
  \(G_i\) to \(G_{i+1}\) for \(i=1,\dots,4\). Let
  \((X_{(12)3},\mu_{(12)3})\) and \((X_{1(23)},\mu_{1(23)})\) have
  meanings as in case of the associativity isomorphism, see
  page~\pageref{page:asso-data}. And let
  \((X_{((12)3)4},\mu_{((12)3)4})\) and other subscripts of \(X\) and
  \(\mu\) with parentheses have similar meanings. Let \(a(\_,\_,\_))\)
  denote the associativity isomorphism when the blanks filled
  appropriately, as discussed for associativity isomorphism. Then the
  associativity coherence demands that the pentagon in
  Figure~\ref{fig:associativity-coherence} should commute.
  \begin{figure}[ht]
    \centering \begin{tikzpicture}[scale=0.9]
      \draw[dar] (0,1.8)--(0,0.2); \draw[dar] (0.1, -0.3)--(2.3,-1.6);
      \draw[dar] (5.7,-0.2)--(3.6,-1.6); \draw[dar]
      (5.8,1.8)--(5.8,0.24); \draw[dar](0.7,2)--(5.3,2);
      \node at (6, 0)[ scale=0.8]{\(X_{1((23)4)}\)}; \node at (6,
      2)[scale=0.8]{\(X_{(1 (23))4}\)}; \node at (0,
      2)[scale=0.8]{\(X_{((12)3)4}\)}; \node at
      (3,-1.7)[scale=0.8]{\( X_{1(2(34))}\)}; \node at (0,
      0)[scale=0.8]{\(X_{(12) (34)}\)};
      \node at ((-1.3, 1)[scale=0.7]{\(a(X_{12}, X_3, X_4)\)}; \node
      at(-0.1, -1.15)[scale=0.7]{\(a(X_1,X_2, X_{34})\)}; \node at
      (6.2,-1.1)[scale=0.7]{\(\Id_{X_1}\times\circ a(X_2, X_3,
        X_4)\)}; \node at (7.2,
      1.0)[scale=0.7]{\(a(X_1, X_{23}, X_4)\)}; \node at (3.0,
      2.25)[scale=0.7]{\(a(X_1,X_2,X_3)\times\Id_{X_4}\)};
    \end{tikzpicture}
    \caption{\small Associativity coherence}
    \label{fig:associativity-coherence}
  \end{figure}

  Let \([[[x_1,x_2],x_3],x_4]\) be a point in \(X
  _{((12)3)4}\). Following the left top vertex of the pentagon along
  the right top sides till the vertex at the bottom, we get that
  \begin{multline*} [[[x_1,x_2],x_3],x_4]\xmapsto{a(X_1,X_2,X_3)\times
      \Id_{X_4}}
    [[x_1,[x_2,x_3]],x_4]\\\xmapsto{a(X_1,X_{23},X_4)}
    [x_1,[[x_2,x_3],x_4]] \xmapsto{\Id_{X_1}\times a(X_2,X_3,X_4)}
    [x_1,[x_2,[x_3,x_4]]].
  \end{multline*}
  And, on the other way,
  \begin{multline*}
    [[[x_1,x_2],x_3],x_4]\xmapsto{a(X_{12},X_3,X_4)}
    [[x_1,x_2],[x_3,x_4]]\xmapsto{a(X_1,X_2,X_{34})}
    [x_1,[x_2,[x_3,x_4]]].
  \end{multline*}
  This the figure commutes.  \medskip

  \noindent viii) {\itshape Identity coherence} : Let
  \((X_i,\lambda_i)\) be topological correspondences from
  \((G_i,\alpha_i)\) to \((G_{i+1},\alpha_{i+1})\) for \(i=1,2\). Let
  \((G_1\circ X_1, \alpha_1\inverse\circ \lambda_2)\) be a composite
  of the identity correspondence at \((G_1,\alpha_1)\) and
  \((X_1,\lambda_1)\); let
  \((X_1\circ G_2, \lambda_1\circ \alpha_2\inverse)\) be a composite
  of \((X_1,\lambda_1)\) and the identity correspondence at
  \((G_2,\alpha_2)\); and let \((X_{12},\lambda_{12})\) be a composite
  of \((X_1,\lambda_1)\) and \((X_2,\lambda_2)\). Then we need to show
  that following diagram is commutative for identity coherence
  commutes:
  \begin{center}
    \begin{tikzpicture}[scale=0.9]
      \draw[dar] (0.5,2)--(3.2,2); \draw[dar] (-0.1,1.8)--(1.7,0.3);
      \draw[dar](4,1.8)--(2.3,0.3); \node at
      (-0.8,2)[scale=1]{\((X_1\circ G_2)\circ X_2\)}; \node at
      (4.5,2)[scale=1]{\(X_1 \circ (G_2\circ X_2)\)}; \node at
      (2,0){\(X_1\circ X_2\)}; \node at (2,
      2.3)[scale=0.8]{\(a(X_1, G_2, X_2)\)}; \node at (0.1,
      0.8)[scale=0.8]{\(r(G_2)\circ \Id_{X_2}\)}; \node at (4.3,
      0.8)[scale=0.8]{\(\Id_{X_1}\circ l(G_1)\)};
    \end{tikzpicture}
  \end{center}
  Let \([[x_1,\gamma], x_2]\in (X_1\circ G_2)\circ X_2\). Then
  \begin{multline*}
    \Id_{X_1}\circ l(G_1) \left(a(X_1, G_2, X_2)\left([[x_1,\gamma],
        x_2]\right)\right) =[x_1, \gamma x_2]\\=[x_1\gamma,
    x_2]=r(G_2)\circ \Id_{X_2}([[x_1,\gamma],x_2]).
  \end{multline*}
  This proves all the axioms and hence the theorem.
\end{proof}

\label{page:asso-proof} From here up to
Remark~\ref{rem:invariance-of-B/B} is the proof of last claim in the
associativity isomorphism in
Theorem~\ref{thm:proof-of-bicategory-of-top-corr}. Recall the data and
notation for the discussion that follows from
page~\pageref{rem:invariance-of-B/B}:

We have the commutating diagram in
Figure~\ref{fig:asso}. From~Lemma~\ref{lem:left-triangle-asso-vertical-measures}
to~\ref{lem:left-triangle-asso-top-measures}, we discuss the measures
residing on the spaces and various families of measures along the maps
in the left-half of Figure~\ref{fig:asso}; the discussion ends at
Lemma~\ref{lem:left-triangle-asso-top-measures}. Following is the idea
behind this: for each \(u\in\base[G]_4\), the measure
\(\lambda_1\times\lambda_2\times{\lambda_3}_u\) on \(\base[T]\) is
\((G_2\times G_3,\alpha_2\times \alpha_3)\)\nb-quasi-invariant. We
multiply this measure by appropriate 0-cocycles on the groupoid \(T\)
so that the resultant measure is invariant. This invariant measure
then agrees a disintegration along the map \(\pi_{123}\) with respect
to a family of measures along \(\pi_{123}\) ---which is induced by
\(\alpha_2 \times\alpha_3\)--- to produce measures \(\mu_{(12)3}\) and
\(\mu_{123}\) on \(X_{123}\).
Lemma~\ref{lem:left-triangle-asso-top-measures} shows that the
measures \({a'_*}\inverse({\mu_{(12)3}}_u)\) and \({\mu_{123}}_u\) on
\(X_{123}\) are equivalent.

In the right-half of Figure~\ref{fig:asso}, a measure
\({\mu_{1(23)}}_u\) similar to \({\mu_{(12)3}}_u\) can be defined on
\(X_{1(23)}\), and an analogue of
Lemma~\ref{lem:left-triangle-asso-top-measures} can be proved for
\({a''_*}\inverse({\mu_{1(23)}}_u)\) and \({\mu_{123}}_u\).

Putting the above two observations together, finally,
Proposition~\ref{prop:left-measures-equi} shows that the families of
measures \(a_*(\mu_{(12)3})\) and \(a_{1(23)}\) on \(X_{1(23)}\) are
equivalent; the proposition also computes that the Radon-Nikodym
derivative \(\dd a_*({\mu_{(12)3}}_u)/\dd {a_{1(23)}}_u\).

\begin{figure}[htb]
  \[
    \begin{tikzcd}[column sep=small]
      & & \base[T] \ar[dl, "\pi_{12}\times \Id_{X_3}"'] \ar[dr, "\Id_{X_1}\times \pi_{23}"] \ar[ddll,"\pi_{(12)3}"', bend right = 50] \ar[dr, "\Id_{X_1}\times \pi_{23}"] \ar[ddrr,"\pi_{(12)3}", bend left  = 50] \ar[dd,"\pi_{123}"]& &\\
      & X_{12}\times_{\base[G]_3}X_3 \arrow[ld, "\pi'_{(12)3}"'] & & X_1\times_{\base[G]_3}X_{23}\arrow[rightarrow]{rd}{\pi'_{(12)3}}&\\
      X_{(12)3} \ar[rrrr, "a={a''}\circ {a'}\inverse", bend right=20]&
      & X_{123}\ar[ll,"{a'}"'] \ar[rr,"{a''}"]& & X_{1(23)}
    \end{tikzcd}
  \]
  \caption{}
  \label{fig:asso}
\end{figure}

Let
\([\alpha_2]\times\delta_{X_3}=
\{[\alpha_2]\times{\delta_{X_3}}_{([x,y],z)}\}_{([x,y],z)\in
  X_{12}\times_{\base[G]_3}X_3}\) be the family of measures along
\(\pi_{12}\times \Id_{X_3}\) defined as
\[
  \int_{\base[T]} f(t)\;\dd
  [\alpha_2]\times{\delta_{X_3}}_{([x,y],z)}(t)=\int_{G_2}
  f(x\gamma,\gamma\inverse y,z)\,\dd\alpha_2(\gamma)
\]
for \(f\in \Contc(\base[T])\); let
\([\alpha'_{3}]=\{{\alpha_3}'_{[[x,y],z]}\}_{[[x,y],z]\in X_{(12)3}}\)
be the one along \(\pi'_{(12)3}\) given by
\[
  \int_{X_{12}\times_{\base[G]_3}X_3}
  g(t)\,\dd[\alpha_3]'_{[[x,y],z]}(t)=\int_{G_3}
  g([x,y]\eta,\eta\inverse z)\,\dd\alpha_3(\eta)
\]
where \(g\in \Contc(X_{12}\times_{\base[G]_3} X_3)\). The composite
\([\alpha'_3]\circ \left([\alpha_2]\times\delta_{X_3}\right)\) gives a
family of measures along
\(\pi_{(12)3}=\pi'_{(12)3}\circ (\pi_{12}\times \Id_{X_3})\).  One the
other hand, define families of measures
\begin{enumerate}[(i)]
\item \(\delta_{X_1}\times[\alpha_{3}]'\) along
  \(\Id_{X_1}\times \pi_{23}\),
\item \([\alpha_2]'\) along \(\pi_{123}'\) and
\item \(\pi_{1(23)}\) along \(\pi_{1(23)}\)
\end{enumerate}
analogous to \([\alpha_2]\times \delta_{X_3},[\alpha_3]'\) and
\(\pi_{(12)3}\), respectively. Finally, let \(\alpha_{123}\) be the
family of measures along \(\pi_{123}\) as in~(2) on
page~\pageref{page:compo-rev-2}.

\begin{definition}
  \label{def:left-vertical-measure}
  Define the following functions
  \begin{align*}
    A_{123}&\colon \Contc(\base[T])\to \Contc(X_{123}),& A_{123}(f)(w)&\defeq {\alpha_{123}}_{w}(f);\\
    A_{12}&\colon \Contc(\base[T])\to \Contc(X_{12}\times_{\base[G]_3}X_3), &
                                                                              A_{12}(f)(p)&\defeq
                                                                                            {\left([\alpha_2]\times\delta_{X_3}\right)}_{p}(f);\\
    A_{(12)3}'&\colon \Contc(X_{12}\times_{\base[G]_3}X_3)\to
                \Contc(X_{(12)3}), &  A_{(12)3}'(h)(q)&\defeq
                                                        {[\alpha_3]'}_{q}(h);\\
    A_{(12)3}&=A_{(12)3}'\circ A_{12} & A_{(12)3}(f)(q), &\defeq {\alpha_{(12)3}}_{q}(f);\\
    {A'}_*&\colon \Contc(X_{123})\to \Contc(X_{(12)3}), & A'_*(k)(w)&\defeq k\circ {a'}\inverse(w)
  \end{align*}
  for
  \(f\in\Contc(\base[T]), h\in X_{12}\times_{\base[G]_3}X_3, k\in
  \Contc(X_{123}), p\in X_{12}\times_{\base[G]_3}X_3, q\in X_{(12)3}\)
  and \(w \in X_{(12)3}\). And define \(A_{23},A_{1(23)}',A_{1(23)}\)
  and \(A_*''\) analogously using the families of measures
  \(\delta_{X_1}\times\alpha_{23},[\alpha_2]',\alpha_{1(23)}\) and the
  homeomorphism \(a''\).
\end{definition}
With the functions in Definition~\ref{def:left-vertical-measure}, we
draw Figure~\ref{fig:asso-measures}.
\begin{figure}[htb]
  \[
    \begin{tikzcd}[column sep=10pt]
      & & \Contc(\base[T]) \ar[dl, "A_{12}"'] \ar[ddll,"{A_{(12)3}}"',
      bend right = 45] \ar[dd,"A_{123}"] \ar[dr, "A_{23}"]
      \ar[ddrr,"{A_{1(23)}}", bend left = 45] & &
      \\
      & \Contc(X_{12}\times_{\base[G]_3}X_3) \arrow[ld, "A'_{(12)3}"']
      & & \Contc(X_1\times_{\base[G]_2}X_{23}) \arrow[rd, "A'_{1(23)}"] &\\
      \Contc(X_{(12)3}) & & \Contc(X_{123}) \ar[ll,"A'_*"']
      \ar[rr,"A''_*"] & & \Contc(X_{1(23)})
    \end{tikzcd}
  \]
  \caption{}
  \label{fig:asso-measures}
\end{figure}

\begin{lemma}
  \label{lem:left-triangle-asso-vertical-measures}
  The maps \(A_*'\) and \(A_*''\) are isomorphisms of complex vector
  spaces. And \(A_{(12)3}=A_*'\circ A_{123}\) and
  \(A_{1(23)}=A_*''\circ A_{123}\).
\end{lemma}
\begin{proof}
  Since \(a'\) and \(a''\) are homeomorphisms, \(A_*'\) and \(A_*''\)
  are isomorphisms of complex vector spaces. Then, for
  \(f\in \Contc(\base[T])\) and \([x,y,z]\in X_{123}\),
  \begin{align*}
    A_{(12)3}(f)([[x,y],z]) &=\int_{G_3} A'_{(12)3}(f)([x,y]\eta,\eta\inverse{z})\,\dd\alpha_3^{r_{X_3}}(\eta)\\
                            &=\int_{G_3} \int_{G_2}
                              f(x\gamma,\gamma\inverse{y}\eta,\eta\inverse{z})\,\dd\alpha_2^{r_{X_2}(y)}(\gamma)\,\dd\alpha_3^{r_{X_3}}(\eta)\\
                            &=\int_{G_3 \times G_2}
                              f(x\gamma,\gamma\inverse{y}\eta,\eta\inverse{z})\,\dd(\alpha_2
                              \times \alpha_3)^{(r_{X_2}(y), r_{X_3}(z))}(\gamma, \eta)\\
                            &=A_{123}(f)([x,y,z])=A'_*(A_{123}(f))([[x,y],z]).
  \end{align*}
  The third equality above is due to Fubini's theorem. This shows that
  \((A_{(12)3}=A_*'\circ A_{123}\). The last claim follows from a
  similar computation as above.
\end{proof}
With the help of Lemma~\ref{lem:left-triangle-asso-vertical-measures},
and definitions of \(A_{(12)3}\) and \(A_{1(23)}\), one can see the
Figure~\ref{fig:asso-measures} commutes.

Now we discuss the (families of) measures on the spaces involved in
the left-half of Figure~\ref{fig:asso}. Let \(i=1,2\) and fix
\(u\in\base[G]_{i+2}\). Recall from the discussion about
Equation~\eqref{equ:composition-1-cocycle} on
page~\pageref{equ:composition-1-cocycle} that the measure
\(\lambda_i\times{\lambda_{i+1}}_u\) on
\(X_i\times_{\base[G]_{i+1}} X_{i+1}\) is
\((G_{i+1},\alpha_{i+1})\)\nb-quasi-invariant; the 1\nb-cocycle
\(D_{i}\) on the transformation groupoid
\((X_i\times_{\base[G]_{i+1}}X_{i+1}) \rtimes G_{i+1}\) that
implements the quasi-invariance and is given by
~Equation~\eqref{equ:composition-1-cocycle}; in this case it is

\begin{equation*}
  D_{i+1}(x_i,x_{i+1},\gamma)=\Delta_{i+1}(\gamma\inverse, x_{i+1})
\end{equation*}
where
\((x_i,x_{i+1},\gamma) \in (X_{i} \times_{\base[G]_{i+1}} X_{i+1})
\rtimes G_{i+1}\) and \(\Delta_{i+1}\) is the adjoining function of
the correspondence \((X_{i+1},\lambda_{i+1})\).

Continuing the above discussion, we know that \(\base[T]\) is a proper
\(G_{i+1}\)\nb-space for an appropriate diagonal action. Now we notice
that the family of measures
\(\{\lambda_1\times\lambda_2\times {\lambda_3}_v\}_{v\in\base[G]_4}\)
on \(\base[T]\) is \((G_{i+1},\alpha_{i+1})\)-quasi-invariant with the
function

\begin{equation}
  \label{eq:T-zero-G-i-q-inv}
  (x_1,x_2,x_3,\gamma)\mapsto D_{i+1}(x_{i},x_{i+1},\gamma)=\Delta_{i+1}(\gamma\inverse,
  x_{i+1}),\quad \base[T]\rtimes G_{i+1}\to \R^+
\end{equation}
as the modular function. Moreover, if we focus on the case \(i=1\),
\[
  b_{12}\times \Id_{X_3}\colon \base[T]\to \R^+
\]
is the 0-cocycle for which the measure
\(b_{12}\times
\Id_{X_3}\cdot(\lambda_1\times\lambda_2\times{\lambda_3}_u)\) is
\(G_2\)-invariant where \(u\in\base[G]_4\). This invariant measure
induces the measure \(\mu_{12}\times{\lambda_3}_u\) on
\(X_{12}\times_{\base[G]_2}X_3\). The action of \(G_3\) on
\(\base[T]\) induces a proper diagonal action on
\(X_{12}\times_{\base[G]_3}X_3\). This is a direct computation. Now
the measure \(\mu_{12}\times{\lambda_3}_u\) is, in turn,
\((G_3,\alpha_3)\)-quasi-invariant with the function
\[
  D_{(12)3}'\colon ([x,y],z,\eta)\mapsto \Delta_3(\eta\inverse,z),
  \quad (X_{12}\times_{\base[G]_3} X_3)\rtimes G_3\to \R^+
\]
as the modular function; this is a direct computation that uses
\((G_3,\alpha_3)\)\nb-quasi-invariance of the family of measures
\(\lambda_3\) on \(X_3\).

Let \(b_{(12)3}'\) be a 0-cocycle on the transformation groupoid
\((X_{12}\times_{\base[G]_3}X_3)\rtimes G_3\) such that
\[
  d^0(b_{(12)3}')=D'_{(12)3}
\]
Then \(b_{(12)3}''(\mu_2\times\lambda_3)\) is a \(G_3\)\nb-invariant
family of measures. This family of measures induces the family of
measures \(\mu_{(12)3}'\) on \(X_{(12)3}\) that so that
\((X_{(12)3},\mu_{(12)3}')\colon (G_1,\alpha_1)\to (G_4,\alpha_4)\) is
a topological correspondence.

All the claims above follow from a direct computation. In the next
lemma and what succeeds it, we discuss the measures on \(X_{123}\)
using the map \(\pi_{123}\).
\begin{lemma}[For associativity isomorphism]
  \label{lem:cocycles-on-T}
  \begin{enumerate}[(i)]
  \item For \(u\in \base[G]_4\), the measure
    \(\lambda_1\times\lambda_2\times{\lambda_3}_u\) on \(\base[T]\) is
    \((T,\alpha_2\times \alpha_3)\)\nb-quasi-invariant. The 1-cocycle
    \(D\) on \(T\) that implements the quasi-invariance is given by
    \( D(x,y,z,\gamma,\eta)\defeq D_2(x,y,\gamma)
    D_{(12)3}'([x,y],z,\eta)\) where \(D_1\) and \(D_2\) are defined
    in above discussion.
  \item The map \(B'\colon \base[T]\to \R^+\) given by
    \(B'\colon (x,y,z)\mapsto b_{12}(x,y)b_{(12)3}'([x,y],z) \) is a
    0\nb-cocycle on \(T\) with \(d^0(B')=D\).
  \end{enumerate}
\end{lemma}
\begin{proof}
  (i): Fix \(u\in\base[G]_4\), and let \(f\in\Contc(T)\). Then
  \begin{multline*}
    \int_{\base[T]}\int_T f(x,y,z,\gamma,\eta)
    \;\dd(\alpha_2\times{\alpha_3})^{(x,y,z)}(\gamma,\eta)\,\dd(\lambda_1\times\lambda_2\times{\lambda_3}_u)(x,y,z)\\
    = \int_{X_3}\int_{X_2}\int_{X_1}\int_{G_3}\int_{G_2}
    f(x,y,z,\gamma,\eta) \,\dd{\alpha_2}^{r_{X_2}(y)}(\gamma)
    \dd{\alpha_3}^{r_{X_3}(z)}(\eta))\\
    \dd\lambda_1^{r_{X_2}(y)}(x)\,\dd\lambda_2^{r_{X_3}(z)}(y)\,\dd\lambda_3^u(z)\\
    = \int_{X_3}\int_{X_2}\int_{X_1}\int_{G_3}\int_{G_2} f(x\gamma,
    \gamma\inverse y\eta, \eta\inverse z
    ,\gamma\inverse,\eta\inverse)\, \,\Delta_{2}(\gamma\inverse,
    \gamma y)\Delta_{3}(\eta\inverse,
    \gamma\inverse z\eta)\\
    \dd{\alpha_2}^{r_{X_2}(y)}(\gamma) \dd{\alpha_3}^{r_{X_3}(z)}(\eta))\,\dd\lambda_1^{r_{X_2}(y)}(x)\lambda_2^{r_{X_3}(z)}(y)\lambda_3^u(z)\\
    = \int_{\base[T]}\int_T f(x\gamma\inverse, \gamma y\eta\inverse,
    \eta z,\gamma,\eta) \, \Delta_{12}(x\gamma\inverse,
    \gamma)\Delta_{23}(\gamma y\eta\inverse, \eta)
    \\\dd(\alpha_2\times{\alpha_3}^{(x,y,z)}(\gamma,\eta))\,
    \dd(\lambda_1\times\lambda_2\times\lambda_3)(x,y,z)\\
    = \int_{\base[T]}\int_T f(x\gamma\inverse, \gamma y\eta\inverse,
    \eta z,\gamma,\eta) \, D(x\gamma\inverse, \gamma y\eta\inverse,
    \eta z,\gamma,\eta)
    \\\dd(\alpha_2\times{\alpha_3}^{(x,y,z)}(\gamma,\eta))\,
    \dd(\lambda_1\times\lambda_2\times\lambda_3)(x,y,z)
  \end{multline*}
  where the third equality is due to a direct computation that usages
  Fubini's theorem and \((G_2,\alpha_2)\)\nb-quasi-invariant of
  \(\lambda_1\times \lambda_2\), and Fubini's theorem and
  \((G_3,\alpha_3)\)\nb-quasi-invariant of
  \(\lambda_2\times \lambda_3\), and Fubini's theorem. This shows that
  the function \(D\) is the desired 0-cocycle.
  
  \noindent (ii): This claim follows from a direct computation: for
  \((x,y,z,\gamma,\eta)\in T\),
  \begin{multline*}
    d^0(B')(x,y,z,\gamma,\eta)\defeq\frac{B'\circ s_T
      (x,y,z,\gamma,\eta)}{ B'\circ r_T
      (x,y,z,\gamma,\eta)}=\frac{B'(x\gamma,\gamma\inverse
      y\eta,\eta\inverse z)}{B'(x,y,z)}\\
    =\frac{b_{12}(x\gamma,\gamma \inverse
      y\eta)}{b_{12}(x,y)}\;\frac{b_{(12)3}'([x\gamma, \gamma \inverse
      y\eta],\eta\inverse z)}{b_{(12)3}'([x,y],z)}
    =\frac{b_{12}(x\gamma,\gamma \inverse
      y)}{b_{12}(x,y)}\;\frac{b_{(12)3}'([x,y]\eta,\eta\inverse
      z)}{b_{(12)3}'([x,y],z)}\\
    =\frac{D_{2}\circ s_{(X_1\times_{\base[G]_2} X_2)\rtimes G_2} (x,
      y, \gamma)}{D_{2}\circ r_{(X_1\times_{\base[G]_2} X_2)\rtimes
        G_2} (x, y, \gamma)}\;\; \frac{D_{(12)3}'\circ
      s_{(X_{12}\times_{\base[G]_3} X_3)\rtimes G_3} ([x,y],z,
      \eta)}{D_{(12)3}'\circ r_{(X_{12}\times_{\base[G]_3} X_3)\rtimes
        G_3}([x,
      y],z,\eta)}\\
    = \Delta_{2}(\gamma\inverse,y)\Delta_{(12)3}([x,y],z\eta) =
    D(x,y,z,\gamma,\eta).
  \end{multline*}
  We use the \(G_3\)\nb-invariance of \(b_{12}\) and
  \(G_2\)\nb-invariance of \(b_{23}\) to get the fourth equality.
\end{proof}

To see what does the analogue of above lemma say regarding the
right-half of Figure~\ref{fig:asso}, we continue the discussion that
follows Equation~\eqref{eq:T-zero-G-i-q-inv} for \(i=2\). In this
case,
\[
  \Id_{X_1}\times b_{23}\colon \base[T]\to \R^+
\]
is the 0-cocycle for which the measure
\(\Id_{X_1} \times b_{23}
\cdot(\lambda_1\times\lambda_2\times{\lambda_3}_u)\) is
\(G_2\)-invariant where \(u\in\base[G]_4\). This invariant measure
induces the measure \({\lambda_1}\times {\mu_{23}}_u\) on
\(X_{1}\times_{\base[G]_2}X_{23}\). The action of \(G_2\) on
\(\base[T]\) induces a proper diagonal action on
\(X_{1}\times_{\base[G]_2}X_{23}\). The measure
\(\lambda_1\times {\mu_{12}}_u\) is, in turn,
\((G_2,\alpha_2)\)-quasi-invariant with the function
\[
  D_{1(23)}''\colon (x,[y,z],\eta)\mapsto \Delta_2(\eta\inverse,y),
  \quad (X_{1}\times_{\base[G]_2} X_{23})\rtimes G_2\to \R^+
\]
as the modular function; recall
from~\cite{Holkar2017Construction-of-Corr}*{Remark 2.5} that
\(\Delta_2\) is \(G_3\)\nb-invariant due to which \(D_{1(23)}''\) is
well-defined.  Now fix a 0-cocycle \(b_{(12)3}''\) on the
transformation groupoid
\((X_{1}\times_{\base[G]_2}X_{23})\rtimes G_2\) with the property that
\[
  d^0(b_{1(23)}')=D'_{1(23)}.
\]
Then \(b_{1(23)}''(\lambda_1\times \mu_2)\) is a \(G_2\)\nb-invariant
family of measures. This family of measures induces the family of
measures \(\mu_{1(23)}'\) on \(X_{1(23)}\) that so that
\((X_{1(23)},\mu_{1(23)}')\colon (G_1,\alpha_1)\to (G_4,\alpha_4)\) is
a topological correspondence. Now the analogue of
Lemma~\ref{lem:cocycles-on-T} can be stated as
\begin{lemma}
  \label{lem:R-cocycles-on-T}
  \begin{enumerate}[(i)]
  \item For \(u\in \base[G]_4\), the measure
    \(\lambda_1\times\lambda_2\times{\lambda_3}_u\) on \(\base[T]\) is
    \((T,\alpha_2\times \alpha_3)\)\nb-quasi-invariant. The 1-cocycle
    \(D_R\) on \(T\) that implements the quasi-invariance is given by
    \( D_R(x,y,z,\gamma,\eta)\defeq D_{1(23)}(x,[y,z],\gamma)
    D_{3}(y,z,\eta)\).
  \item The map \(B''\colon \base[T]\to \R^+\) given by
    \(B''\colon (x,y,z)\mapsto b_{23}(y,z)b_{1(23)}''(x,[y,z]) \) is a
    0\nb-cocycle on \(T\) with \(d^0(B'')=D_R\).
  \end{enumerate}
\end{lemma}
\begin{proof}
  Similar to that of Lemma~\ref{lem:cocycles-on-T}.
\end{proof}

\begin{remark}
  Since the cocycles \(D\) and \(D_R\) in
  Lemmas~\ref{lem:cocycles-on-T} and~\ref{lem:R-cocycles-on-T} are
  modular functions of the measure
  \(\lambda_1\times\lambda_2\times\lambda_3\) on the unit space of the
  locally compact groupoid \(T\) equipped with the Haar system
  \(\alpha_1\times\alpha_2\), we have \(D=D_R\)
  \(\lambda_1\times\lambda_2\times\lambda_3\)-almost everywhere on
  \(\base[T]\). But both \(D\) and \(D_R\) are continuous. Therefore,
  \(D=D_R\). Hereon, we shall use \(D\) for this cocycle. In fact, one
  can show that \(D(x,y,z,\gamma,\eta)=D_2(x,y,\gamma)D_3(y,z,\eta)\)
  and \(P\colon \base[T]\to \R^+\) defined by
  \(P(x,y,z)=b_{12}(x,y)b_{23}(y,z)\) is a 0-cocycle with
  \(d^0(P)=D\).
\end{remark}

Now we fix a 0-cocycle \(B\) on \(T\) with the property that
\(d^0(B)=D\); use the same cocycle while working with the right-half
of Figure~\ref{fig:asso}. Then Lemma~\ref{lem:cocycles-on-T} and
Proposition~{3.1}{(i)} in~\cite{Holkar2017Composition-of-Corr} gives
us a family of measures
\(\mu_{123}\defeq \{{\mu_{123}}_u\}_{u\in\base[G]_4}\) with the
property that
\begin{equation}\label{eq:measure-on-X-123} {\mu_{123}}_u(A_{123}(f))=
  \left(B\cdot
    (\lambda_1\times\lambda_2\times{\lambda_3}_u)\right)(f)
\end{equation}

for all \(u\in\base[G]_4\). Additionally, since \(d^0(B')=d^0(B)=D\),
Remark~{1.1.4} in~\cite{Holkar2017Construction-of-Corr} says that the
function \(B'/B\) (or \(B''/B\)) is constant on the \(T\)\nb-orbits of
\(\base[T]\). Thus \(B'/B\) (or \(B''/B\)) induces a continuous
function on \(\base[T]\7T=X_{123}\) which we denote by \([B'/B]\) (or
\([B''/B]\), respectively).
\begin{lemma}
  \label{lem:left-triangle-asso-top-measures}
  The families of measures \({a'}\inverse_*\mu_{(12)3}\) and
  \(\mu_{123}\) on \(X_{123}=\base[T]/T\) are equivalent, and
  \([B'/B]\) is the Radon-Nikodym derivative
  \(\dd {a'}\inverse_*{\mu_{(12)3}}_u/\dd{\mu_{123}}_u\) where
  \(u\in\base[G]_4\).
\end{lemma}
\begin{proof}
  We claim that the measure \({a'}_*\inverse({\mu_{(12)3}}_u)\) on
  \(X_{123}\) disintegrates the measure
  \(B'\cdot\lambda_1\times\lambda_2\times{\lambda_3}_u\)
  \[ {a'}_*\inverse({\mu_{(12)3}}_u)\circ A_{123}
    B'\cdot\lambda_1\times\lambda_2\times{\lambda_3}_u.
  \]
  Therefore, \(B'\cdot\lambda_1\times\lambda_2\times{\lambda_3}_u\) is
  a \((T,\alpha_2\times\alpha_3)\)-invariant measure on \(\base[T]\)
  (see Proposition 3.1(i) in~\cite{Holkar2017Composition-of-Corr}). We
  already know that
  \(B\cdot\lambda_1\times\lambda_2\times{\lambda_3}_u\) is also such
  an invariant measure on \(\base[T]\), see
  Equation~\eqref{eq:measure-on-X-123} and the discussion preceding
  it. Moreover, since \(B\) and \(B'\) are \(\R^+\)\nb-valued
  cocycles, \(B'\cdot\lambda_1\times\lambda_2\times{\lambda_3}_u\) and
  \(B\cdot\lambda_1\times\lambda_2\times{\lambda_3}_u\) are equivalent
  measures with the Radon-Nikodym derivative
  \[
    \frac{\dd\;
      B'\cdot\lambda_1\times\lambda_2\times{\lambda_3}_u}{\dd\;
      B\cdot\lambda_1\times\lambda_2\times{\lambda_3}_u} =
    \frac{B'}{B}.
  \]
  Now we apply Lemma~\ref{lem:equi-measures-on-quotient} to these two
  invariant measures on \(\base[T]\) to infer the claim of the present
  lemma. Following is the proof regarding disintegration of the
  measure: Let \(f\in\Contc(\base[T])\) and \(u\in\base[G]_4\). Then
  \begin{multline*}
    {a'}\inverse_*{\mu_{(12)3}}_u(A_{123}(f))=\defeq
    {\mu_{(12)3}}_u(A_{123}(f)\circ {a'}\inverse)\\=
    {\mu_{(12)3}}_u(A'_*\circ A_{123}(f))=
    {\mu_{(12)3}}_u(A_{(12)3}(f))
    ={\mu_{(12)3}}_u\left(A_{23}\left(A_{12}(f)\right)\right)
  \end{multline*}
  where the first equality is due to
  Definition~\ref{def:left-vertical-measure}, second one is the
  definition of push-forward of a measures, the second one follows due
  to Lemma~\ref{lem:left-triangle-asso-vertical-measures} and third
  one follows from the definition of \(A_{(12)3}\)
  (cf.\,Definition~\ref{def:left-vertical-measure}). We use the
  disintegration
  \(b_{(12)3}'\cdot\mu_{12}\times{\lambda_3}_u={\mu_{(12)3}}_u\circ
  A_{23}\) of measures along \(\pi_{(12)3}'\) to see that the last
  term above equals
  \begin{align*}
    &b_{(12)3}'\cdot\mu_{12}\times{\lambda_3}_u\left(A_{12}(f)\right)
    \\
    &=\int_{X_3}\int_{X_{12}}  A_{12}(f) ([x,y],z)
      b_{(12)3}' ([x,y],z)\,\dd{\mu_{12}}_{r_{X_3}(z)}([x,y])\,\dd{\lambda_{3}}_u(z)\\
    &=\int_{X_3} {\mu_{12}}_{r_{X_3}(z)}\left(A_{12}(f) 
      b_{(12)3}'\right)\,\dd{\lambda_{3}}_u(z).
  \end{align*}
  Now we use the disintegration of families of measures
  \(\mu_{12}\circ A_{12}=b_{12}\cdot \lambda_1\times\lambda_2\) to see
  that the last term in above equation equals
  \begin{multline*}
    \int_{X_3}\int_{X_1}\int_{X_{2}} f(x,y,z) b_{12}(x,y) b_{(12)3}'
    ([x,y],z)\,\dd{\lambda_1}_{r_{X_2}(y)}(x)\,\dd{\lambda_{2}}_{r_{X_3}(z)}(y)\,\dd{\lambda_{3}}_u(z)\\
    =\int_{X_1}\int_{X_2} \int_{X_3} f(x,y,z) B'(x,y,z)\,
    \dd{\lambda_1}_{r_{X_2}(y)}(x)\,\dd{\lambda_{2}}_{r_{X_3}(z)}(y)\,\dd{\lambda_{3}}_u(z)
  \end{multline*}
  where \(B'\) is the 0-cocycle on \(T\) as in
  Lemma~\ref{lem:cocycles-on-T}.
\end{proof}

\begin{proposition}
  \label{prop:left-measures-equi}
  The families of measures \(a_*(\mu_{(12)3})\) and \(\mu_{1(23)}\) on
  \(X_{1(23)}\) are equivalent, and the Radon-Nikodym derivative
  \( \dd a_*{\mu_{(12)3}}/\dd {\mu_{1(23)}}=[B'/B'']\circ
  {a''}\inverse\).
\end{proposition}
\begin{proof}
  Lemma~\ref{lem:left-triangle-asso-top-measures} shows that
  \({a'}\inverse_*(\mu_{(12)3})\sim\mu_{123}\) on \(X_{123}\) and the
  Radon--Nikodym derivative
  \(\dd {a'}\inverse_*{\mu_{(12)3}}_u/\dd{\mu_{123}}_u=[B'/B]\) for
  \(u\in\base[G]_4\). For the right-half of Figure~\ref{fig:asso}, one
  may prove ---on the similar lines as
  Lemma~\ref{lem:left-triangle-asso-top-measures}--- that
  \({a''}\inverse_*(\mu_{1(23)})\sim\mu_{123}\) on \(X_{123}\) and the
  Radon--Nikodym derivative
  \(\dd {a''}\inverse_*{\mu_{1(23)}}_u/ \dd{\mu_{123}}_u=[B''/B]\) for
  \(u\in\base[G]_4\).

  Now the transitivity of equivalence of measures implies that
  \({a'}\inverse_*(\mu_{(12)3})\sim {a''}\inverse_*(\mu_{1(23)})\) on
  \(X_{123}\), and the function \([B'/B'']\) implements the
  equivalence. Using Lemma~\ref{lem:equi-measures-image}, we see that
  \begin{equation}
    a''_*{a'}\inverse_*(\mu_{(12)3})\sim a''_*{a''}\inverse_*(\mu_{1(23)})
    \quad \text{ on } X_{1(23)}.\label{eq:left-triangle-measure}
  \end{equation}
  Furthermore, Lemma~\ref{lem:equi-functoriality}, implies that
  \(a''_*{a'}\inverse_*(\mu_{(12)3})=(a''\circ{a'}\inverse)_*(\mu_{(12)3})=a_*(\mu_{(12)3})\),
  and, similarly, \(a''_*{a''}\inverse_*(\mu_{1(23)})
  =\mu_{(12)3}\). Therefore, Equation~\eqref{eq:left-triangle-measure}
  says that \(a_*(\mu_{(12)3})\sim\mu_{1(23)}\).  Moreover, due to the
  Chain rule, for each \(u\in\base[G]_4\) the Radon-Nikodym derivative
  \[
    \dd a_*{\mu_{(12)3}}_u/ \dd{\mu_{1(23)}}_u=[B'/B'']\circ
    {a''}\inverse.
  \]
\end{proof}

\begin{remark}
  \label{rem:invariance-of-B/B}
  Since \(B',B''\) and \(a''\) are \(G_4\)\nb-invariant, so are
  \(B'/B'', [B'/B'']\) and \([B'/B'']\circ {a''}\inverse\). The
  0\nb-cocycles \(B\) and \(B'\) on \(\base[T]\) induces well-defined
  function \([B'/B'']\) on \(X_{123}\).
\end{remark}

\subsection{The $\Cst$-bifunctor}
\label{sec:cst-bifunctor}

\begin{lemma}
  \label{lemma:Deltas-of-isomorphic-corr}
  Let \((X,\lambda)\) and \((Y,\tau)\) be topological correspondences,
  with \(\Delta_X\) and \(\Delta_Y\) as the adjoining functions,
  respectively, from \((G,\alpha)\) to \((H,\beta)\). Let
  \(\textup{t}\colon X\rightarrow Y\) be an isomorphism of topological
  correspondences. Let \(G\ltimes X\) denote the transformation
  groupoid for the left action of \(G\) on \(X\).  Let
  \(M\colon Y\to \R\) be the continuous function
  \(M(y)= \frac{\tau_{s_Y(y)}}{\dd\textup{t}_*(\lambda_{s_Y(y)})}
  (y)\).
  \begin{enumerate}[label=\roman*), leftmargin=*]
  \item \(M\) is \(H\)\nb-invariant.
  \item
    \(\Delta_X=\left(M\circ \textup{t}\circ r_{G\ltimes X}/M\circ
      \textup{t} \circ s_{G\ltimes X}\right)\;\Delta_Y\circ(\Id\times
    \textup{t})\).
  \end{enumerate}
\end{lemma}
\begin{proof}
  (i): Fix \(u\in \base[H]\). By definition \(\tau_u\) is invariant
  measures on the space of units of the transformation groupoid
  \(X\rtimes H\) when the transformation groupoid is equipped with the
  Haar system induced by \(\beta\); similar claim holds for
  \(\lambda_u\). Since \textup{t} is \(G\)-\(H\)\nb-equivariant
  homeomorphism, \(\textup{t}_*(\lambda_u)\) is also a
  \(X\rtimes H\)\nb-invariant measure on \(Y\). Also, since t is an
  isomorphism of topological correspondences
  \(\textup{t}_*(\lambda_u)\sim \tau_u\) for all \(u\in\base[H]\).
  Now Lemma~\ref{lem:equi-measures-on-quotient} gives us the desired
  result.

  \noindent (ii): The isomorphism \textup{t} induces the
  \(G\)-\(H\)\nb-equivariant homeomorphism
  \(\Id \times \textup{t}\colon G \times_{\base[G]} X\to
  G\times_{\base[G]} Y\). Let \(f\in \Contc(G\ltimes X)\) and
  \(u\in\base\). Then we may write
  \begin{multline}\label{eq:relation-between-Adj-fns-1}
    \int_X\int_G
    f(\gamma\inverse,x)\,\dd\alpha^{r_X(x)}(\gamma)\,\dd\lambda_u(x)
    \\= \int_X\int_G f(\gamma\inverse,t\inverse(y))
    \,\dd\alpha^{r_Y(y)}(\gamma) \,\dd\lambda_u(\textup{t}\inverse(y))
  \end{multline}
  as \(r_Y(y)=r_X(x)\). The last term can be written as
  \[
    \int_X\int_G f\circ (\Id\times t)\inverse(\gamma\inverse,y)
    \,\dd\alpha^{r_Y(y)}(\gamma) \,\dd\textup{t}_*(\lambda_u)(y).
  \]
  Change the measures \(\textup{t}_*(\lambda_u)\) to \(\tau_u\) makes
  above term equal to
  \[
    \int_X\int_G f\circ (\Id\times t)\inverse (\gamma\inverse,
    y)\,\frac{\dd\textup{t}_*(\lambda_u)}{\dd\tau_u}(y)\,\dd\alpha^{r_Y(y)}(\gamma)
    \,\dd\tau_u(y).
  \]
  Using the \((G,\alpha)\)\nb-quasi-invariance of \(\tau\), we see
  that above term equals
  \[
    \int_X\int_G f\circ (\Id\times t)\inverse (\gamma,\gamma\inverse
    y)\,\Delta_Y(\gamma,\gamma\inverse y)\,
    \frac{\dd\textup{t}_*(\lambda_u)}{\dd\tau_u}
    (y)\,\dd\alpha^{r_Y(y)}(\gamma) \,\dd\tau_u(y).
  \]
  Now change the measures \(\tau_u\) to \(\textup{t}_*(\lambda_u)\) to
  see that above term equals
  \[
    \int_X\int_G f\circ (\Id\times t)\inverse (\gamma,\gamma\inverse
    y)\,\Delta_Y(\gamma,\gamma\inverse y)\,
    \frac{\dd\textup{t}_*(\lambda_u)}{\dd\tau_u}(y)
    \frac{\dd\tau_u}{\dd\textup{t}_*(\lambda_u)} (\gamma\inverse
    y)\,\dd\alpha^{r_Y(y)}(\gamma) \,\dd \textup{t}_*(\lambda_u)(y).
  \]
  Now note that, since \(\textup{t}_*(\lambda_u)\sim\tau_u\),
  \(\dd\tau_u/\dd\textup{t}_*(\lambda_u)=(\dd\textup{t}_*(\lambda_u)/\dd\tau_u)^{-1}\). Therefore,
  above term equals which is same as
  \begin{equation}\label{eq:relation-between-Adj-fns-2}
    \int_X\int_G f(\gamma,\gamma\inverse x)\,\Delta_Y(\gamma,\gamma\inverse \textup{t}(x))\,\frac{\dd\textup{t}_*(\lambda_u)}{\dd\tau_u}(y) \left(\frac{\dd\textup{t}_*(\lambda_u)}{\dd\tau_u} (\gamma\inverse
      y)\right)^{-1}\,\dd\alpha^{r_X(x)}(\gamma) \,\lambda_u(x).
  \end{equation}
  Comparing Equations~\eqref{eq:relation-between-Adj-fns-1}
  and~\eqref{eq:relation-between-Adj-fns-2} with the definition of
  adjoining function of \((X,\lambda)\), we get
  \begin{align*}
    \Delta_X(\gamma\inverse,x) &= \left.\Delta_Y(\gamma\inverse,
                                 \textup{t}(x))\,\frac{\dd\textup{t}_*(\lambda_u)}{\dd\tau_u}(\gamma\inverse
                                 y) \middle/
                                 \frac{\dd\textup{t}_*(\lambda_u)}{\dd\tau_u}
                                 (y)\right.\\
                               &=\frac {M\circ \textup{t}\circ
                                 r_{G\ltimes X}}{M\circ \textup{t}\circ
                                 s_{G\ltimes X}} (\gamma\inverse, x)\;
                                 \Delta_Y\circ (\Id\times \textup{t})
                                 (\gamma\inverse, x)
  \end{align*}
  \(\lambda_u\circ\alpha\)\nb-almost everywhere on
  \(G\times_{\base[G]}X_u\). But \(\Delta_X, \Delta_Y, \textup{t}\)
  and \(M_u\) are continuous functions, therefore, the equality holds
  for all \((\gamma\inverse,x)\in G\times_{\base}X\).
\end{proof}

\begin{proposition}\label{prop:vertical-functoriality}
  With the same data and hypothesis as
  Lemma~\ref{lemma:Deltas-of-isomorphic-corr}, \(\textup{t}\) induces
  an isomorphism \(\textup{T}\colon\Hils(X,\lambda)\to\Hils(Y,\tau)\)
  of \(\Cst\)\nb-correspondences.
\end{proposition}
\begin{proof}
  Let \(M\) continue to have the same meaning in
  Lemma~\ref{lemma:Deltas-of-isomorphic-corr}. Define
  \[
    \textup{T}\colon \Contc(X)\rightarrow \Contc(Y),\quad
    \textup{T}(f)=f\circ \textup{t}\inverse \cdot M^{1/2}
  \]
  where \(f \in \Contc(X)\). Then we claim that \(\textup{T}\) induces
  an isomorphism of topological correspondences from
  \(\Hils(X,\lambda)\) to \(\Hils(Y,\tau)\).

  We first prove that \(\textup{T}\) extends to a unitary operator of
  Hilbert \(\Cst(H,\beta)\)\nb-modules
  \(\Hils(X,\lambda)\to\Hils(Y,\tau)\). Firstly, \(\textup{T}\) is
  clerly a linear mapping between complex vector spaces. Let
  \(\psi\in \Contc(H)\) and \(f, g\in \Contc(X)\). Then
  \begin{align*}
    \textup{T}(f\psi)(y)
    &=(f\psi)(\textup{t}\inverse (y)) \, M^{1/2}(y)=\int_H f(\textup{t}\inverse(y)\eta)\psi(\eta\inverse) \, M^{1/2}(y) \,\dd\beta^{s_X(x)}(\eta)\\
    &=\int_H f\circ\textup{t}\inverse(y\eta) f\circ\textup{t}\inverse(y\eta)\psi({\eta}\inverse) \, M^{1/2}(y\eta) \,\dd\beta^{s_X(x)}(\eta)\\
    &=\int_H \textup{T}(f)(y\eta)\psi(\eta\inverse)\,\beta^{s_Y(y)}(\eta).
  \end{align*}
  To get the third equality above, we use the facts that
  \(\textup{t}\inverse\) is a \(H\)\nb-equivariant homeomorphism, and
  that \(M\) is \(H\)\nb-invariant
  (cf.\,\ref{lemma:Deltas-of-isomorphic-corr}{(i)}). Thus
  \(\textup{T}\) is a linear map of pre-Hilbert modules over the
  pre-\(\Cst\)\nb-algebra \(\Contc(H)\).

  Now, recall from Lemma~\ref{lemma:Deltas-of-isomorphic-corr} that
  \(M=\dd\textup{t}_*(\lambda)/\dd\tau\). Applying chain rule to the
  composites of the functions
  \((X,\lambda_u) \xrightarrow{\textup{t}} (Y,\tau_u)
  \xrightarrow{\textup{t}\inverse} (X,\lambda_u)\) of measures spaces,
  where \(u\in\base[H]\), one observes that
  \begin{equation}\label{eq:adjo-of-T}
    1=\dd\lambda_u/\dd\lambda_u=
    M\circ \textup{t}\cdot
    \frac{\dd\textup{t}\inverse_*(\tau_u)}{\dd\lambda_u} \quad\text{ or }\quad \frac{ \dd\textup{t}\inverse_*
      (\tau_u)}{\dd\lambda_u} = \frac{\dd \tau_u}{\dd\textup{t}_*(\lambda_u)}\circ \textup{t}=\frac{1}{M}\circ \textup{t}.
  \end{equation}
  This observation motivates us to define
  \[
    \textup{T}^* \colon \Contc(Y) \rightarrow \Contc(X)
    \quad\text{by}\quad\textup{T}^*(g) = g\circ \textup{t} \cdot
    \sqrt{\frac{\dd\tau}{\dd\textup{t}_*(\lambda)}}\circ\textup{t}
    =g\circ \textup{t}\cdot\frac{1}{\sqrt{M}}\circ \textup{t}
  \]
  for \(g\in \Contc(Y)\), and expect that \(\textup{T}^*\) is the
  adjoint of T. We verify claim: for
  \(f\in \Contc(X), g\in \Contc(Y)\) and \(\eta\in H\),
  Equation~\eqref{def:inner-product} gives
  \begin{equation*}
    \inpro{\textup{T}f}{g}(\eta)
    = \int_X\overline{f\circ t\inverse (y)} M^{1/2}(y)
    g(y\eta)\,\dd\tau_{r_H(\eta)}(y).
  \end{equation*}
  Now we change the measure \(\tau_{r_H(\eta)}\) to
  \(\textup{t}_*(\lambda_{r_H(\eta)})\) in the last term so it becomes
  \begin{multline*}
    \int_X\overline{f\circ t\inverse (y)}\,
    \sqrt{\frac{\dd\textup{t}_*(\lambda_{r_H(\eta)})}{\dd\tau_{r_H(\eta)}}(y)}\frac
    {\dd\tau_{r_H(\eta)}}{\dd \textup{t}_*(\lambda_{r_H(\eta)})}(y)
    g(y\eta)\,\dd\textup{t}_*(\lambda_{r_H(\eta)})(y)\\=
    \int_X\overline{f\circ t\inverse (y)}\, \sqrt{\frac
      {\dd\tau_{r_H(\eta)}}{\dd \textup{t}_*(\lambda_{r_H(\eta)})}}(y)
    g(y\eta)\,\dd\textup{t}_*(\lambda_{r_H(\eta)})(y).
  \end{multline*}
  Recall that t is an equivariant homeomorphism. Now change the
  measure \(\textup{t}_*(\lambda_{r_H(\eta)})\) to
  \(\textup{t}\inverse_*(\textup{t}_*(\lambda_{r_H(\eta)}))=\lambda_{r_H(\eta)}\).
  Then \(y\mapsto \textup{t}(x)\) and the last term above becomes
  \begin{multline*}
    \int_X\overline{f(x)}\, \sqrt{\frac {\dd\tau_{r_H(\eta)}}{\dd
        \textup{t}_*(\lambda_{r_H(\eta)})}}\circ \textup{t}(x) g\circ
    \textup{t}(x\eta)\,\dd\lambda_{r_H(\eta)}(x)\\=
    \int_X\overline{f(x)}\,
    \textup{T}^*(g)(x\eta)\,\dd\lambda_{r_H(\eta)}(x)=\inpro{f}{T^*(g)}(\eta).
  \end{multline*}
  Moreover, using Equation~\eqref{eq:adjo-of-T} one can check that
  \(\textup{Id}_{\Contc(X)}=\textup{T}^*\circ \textup{T}\) and
  \(\textup{Id}_{\Contc(Y)}=\textup{T}\circ \textup{T}^*\) --- for
  \(f\in \Contc(X)\) and \(x\in X\),
  \[
    \textup{T}^*(\textup{T}(f))(x)= \textup{T}(f)(t(x))\cdot
    \frac{1}{\sqrt{M(\textup{t}(x))}} = f(x)
    \frac{\sqrt{M(\textup{t}(x))}}{\sqrt{M(\textup{t}(x))}} = f(x).
  \]
  Similarly, \(\textup{T}\circ
  \textup{T}^*=\textup{Id}_{\Contc(Y)}\). This implies that
  \(\textup{T}\colon \Contc(X)\to \Contc(Y)\) is an isometric
  isomorphism isomorphism of pre-Hilbert modules over the
  pre-\(\Cst\)\nb-algebra \(\Contc(H,\beta)\) with
  \(\textup{T}\inverse = \textup{T}^*\). Similar claim holds for
  \(\textup{T}^*\). Now extend the isometries T and \(\textup{T}^*\)
  to unitaries of Hilbert module \(\Hils(X,\lambda)\) and
  \(\Hils(Y,\tau)\); we shall use the symbols T and \(\textup{T}^*\)
  themselves for the unitary extensions of the isomorphisms T and
  \(\textup{T}^*\) of pre-Hilbert modules.  \smallskip
  
  In the rest of the part of the proof, we show that T\('\)
  intertwines the representation of \(\Cst(G,\alpha)\) on
  \(\Hils(X,\lambda)\) and \(\Hils(Y,\tau)\). This will complete the
  proof.

  Let
  \begin{equation*}
    \pi_1 \colon \Cst(G,\alpha)\to
    \Bound(\Hils(X,\lambda))_{\Cst(H,\beta)}\quad\text{and}\quad
    \pi_2 \colon \Cst(G,\alpha)\to\Bound(\Hils(Y,\tau))_{\Cst(H,\beta)}
  \end{equation*}
  denote the representations that gives the
  \(\Cst\)\nb-correspondences \(\Hils(X,\lambda)\) and
  \(\Hils(Y,\tau)\). The first one in
  Equation~\eqref{def:left-right-action} gives these representations
  for certain functions. Due to density of \(\Contc(X)\) and
  \(\Contc(Y)\) in \(\Hils(X,\lambda)\) and \(\Hils(Y,\tau)\), it
  suffices to show that
  \(\textup{T}\circ\pi_1(\psi)(f)=\pi_2(\psi)\circ\textup{T}(f)\) for
  \(\psi\in \Contc(G)\) and \(f\in \Contc(X)\). Following is the
  computation:
  \begin{align*}
    \pi_2(\psi)(\textup{T}(f))(y) &=\int_G \psi(\gamma)\textup{T}(f)(\gamma\inverse y)\,\Delta_Y^{1/2}(\gamma,\gamma\inverse y)\,\dd\alpha^{r_Y(y)}(\gamma)\\
                                  &=\int_G \psi(\gamma)f\circ \textup{t}\inverse (\gamma\inverse
                                    y)\,M^{-1/2}(\gamma\inverse y)\,\Delta_Y(\gamma,\gamma\inverse
                                    y)^{1/2}\,\dd\alpha^{r_Y(y)}(\gamma)
  \end{align*}
  where \(y\in Y\).  The last term above can also be written as
  \[
    \int_G \psi(\gamma)f\circ \textup{t}\inverse (\gamma\inverse
    y)\frac {M^{1/2}(y)}{M^{1/2}(\gamma\inverse
      y)}\,\Delta_Y(\gamma,\gamma\inverse y)^{1/2}
    M^{-1/2}(y)\,\dd\alpha^{r_Y(y)}(\gamma).
  \]
  Now change the variable \(y\mapsto t(x)\), and use the
  \(G\)\nb-equivariance of \(\textup{t}\inverse\) so the last term can
  be written as
  \[
    \int_G \psi(\gamma)f (\gamma\inverse
    x)\sqrt{\left(\frac{M\circ\textup{t}(x)}{M\circ
          \textup{t}(\gamma\inverse x)}\,\Delta_Y\circ (\Id_G\times
        \textup{t})(\gamma,\gamma\inverse x)\right)}
    M^{-1/2}(\textup{t}(x))\,\dd\alpha^{r_X(x)}(\gamma).
  \]
  Using Lemma~\ref{lemma:Deltas-of-isomorphic-corr}(ii), we can write
  the above term as
  \begin{align*}
    &\int_G \psi(\gamma)f(\gamma\inverse
      x)\sqrt{\Delta_X(\gamma,\gamma\inverse x)}
      M^{-1/2}(\textup{t}(x))\,\dd\alpha^{r_X(x)}(\gamma)\\
    &=M^{-1/2}(\textup{t}(x))\int_G \psi(\gamma)f(\gamma\inverse
      x)\Delta_X^{1/2}(\gamma,\gamma\inverse x)
      \,\dd\alpha^{r_X(x)}(\gamma)\\
    &= M^{-1/2}(\textup{t}(x)) \cdot\pi_2(\psi)
      (f)(x)=M^{-1/2}(y) \cdot\left(\pi_1(\psi)
      f\right)\circ \textup{t}\inverse (y)= T\circ \pi_1(\psi)(f)(y);
  \end{align*}
  in the second last equality above, we change the variable
  \(\textup{t}(x)\mapsto y\).
\end{proof}

Note that in Proposition~\ref{prop:vertical-functoriality}, the
identity map on \(X\) induces the identity isomorphism on
\(\Hils(X,\lambda)\).

\begin{corollary}
  \label{cor:functoriality-on-vetical-arrows}
  Along with the same data and hypothesis as
  Lemma~\ref{lemma:Deltas-of-isomorphic-corr}, assume that
  \((Z,\kappa)\) is another topological correspondence from
  \((G,\alpha)\to (H,\beta)\), and \(\textup{l}\colon Y\to Z\) is an
  isomorphism of correspondences. If
  \(\textup{L}\colon \Hils(Y,\tau)\to \Hils(Z,\kappa)\) is the
  isomorphism of \(\Cst\)\nb-correspondences that \(l\) induces, and T
  the one induced by t, then
  \(\textup{T}\circ\textup{L}\colon \Hils(X,\lambda)\to
  \Hils(Z,\kappa)\) is the isomorphism of \(\Cst\)\nb-correspondences
  that the composite \(l\circ t\) induces.
\end{corollary}
\begin{proof}
  Follows from the definitions of T and L and the chain rule for
  measures.
\end{proof}

\begin{remark}
  Let \((X,\mathfrak{B},\mu)\) be seperable \(\sigma\)\nb-finite
  measure space. In~\cite[Definition
  2.2]{Sunder1996ImprimitivityNotes}, Sunder defines an automorphism
  of \((X,\mathfrak{B},\mu)\) as a \(\mathfrak{B}\)\nb-measurable
  function \(T\colon X\to X\) which is \(\mu\)\nb-almost everywhere
  invertible. Then the automorphisms of \((X,B,\mu)\) form a group,
  say \(G\). Let \(\Hils\) be a separable Hilbert space, and let
  \(\Ltwo(X,\mu; \Hils)\) denote the Hilbert space of
  \(\mu\)\nb-square integrable function on \(X\) taking values in
  \(\Hils\). Then Sunder shows the \(G\) group has an obvious unitary
  representation on \(\Ltwo(X,\mu;\Hils)\), see~\cite[Proposition
  2.4]{Sunder1996ImprimitivityNotes} for details and compare it with
  Proposition~\ref{prop:vertical-functoriality} and
  Corollary~\ref{cor:functoriality-on-vetical-arrows}. On the similar
  lines, with the help of
  Proposition~\ref{prop:vertical-functoriality} and
  Corollary~\ref{cor:functoriality-on-vetical-arrows}, one may start
  with a locally compact second countable topological space \(X\), a
  map \(\pi\colon X\to Y\) of spaces, a continuous family of measures
  along \(\pi\). Define the group \(G\) of continuous automorphisms of
  the pair \(\mu\) as all homeomorphisms \(T\) of \(X\) with
  \(\pi\circ T=\pi\). Now let \(p\colon \Hils\to X\) be a continuous
  field of Hilbert spaces. Let \(\Ltwo(X,\mu; \Hils)\) denote the
  Hilbert \(\Contz(Y)\)\nb-module consisting of square-integrable
  sections of \(p\). Then \(G\), the group of automorphisms of
  \((X,\mu)\), has a representation on \(\Ltwo(X,\mu;\Hils)\). This
  representation is given by similar formula as in~\cite[Proposition
  2.4]{Sunder1996ImprimitivityNotes}.
\end{remark}

\begin{corollary}
  Let
  \begin{align*}
    (X, \alpha)&\colon (G_1, \lambda_1) \rightarrow (G_2, \lambda_2),\\
    (Y, \beta)&\colon (G_2, \lambda_2) \rightarrow (G_3, \lambda_3)
  \end{align*}
  be correspondences and let
  \((\Omega, \mu), (\Omega, \mu')\colon (G_1, \lambda_1)\rightarrow
  (G_3, \lambda_3)\) be two composites of them. Then
  \(\Hils(\Omega,\mu)\) and \(\Hils(\Omega,\mu')\) are isomorphic
  \(\Cst\)\nb-correspondences.
\end{corollary}
\begin{proof}
  Follows directly from Proposition~\ref{prop:lifted-cocycles-are-iso}
  and Proposition~\ref{prop:vertical-functoriality}.
\end{proof}
\medskip

Denote the bicategory of topological correspondences by \(\tcorr\) and
the one of \(\Cst\)\nb-correspondences by \(\ccorr\).
\begin{theorem}
  \label{thm:functoriality}
  Let \((G,\alpha),(H,\beta)\) objects in \(\tcorr\) and
  \((X,\lambda)\colon (G,\alpha)\to (H,\beta)\) a 1\nb-arrow. Then the
  assignments \((G,\alpha)\mapsto \Cst (G,\alpha)\) and
  \((X,\lambda)\mapsto \Hils(X,\lambda)\) define a bifunctor from
  \(\tcorr\) to \(\ccorr\).
\end{theorem}
\begin{proof}
  Recall Definition~\ref{def:bifunctor}. We define the bifunctor
  \(\mathfrak{F}=(F,\phi)\colon \tcorr \to \ccorr\) by the following
  assignments of objects, 1\nb-arrows, 2\nb-arrows, identity morphisms
  and natural transformations in \(\tcorr\) to that of \(\ccorr\). The
  note in parentheses at the end of each item indicates what data in
  the definition of bifunctor it referes to.
  \begin{description}[font=\normalfont\itshape]
  \item[Object:] \(F((G,\alpha))=\Cst(G,\alpha)\) (Data (i) in
    Definition~\ref{def:bifunctor})
  \item[1-arrow:] map a 1\nb-arrow
    \((X,\lambda)\colon (G,\alpha)\to (H,\beta)\) to the 1\nb-arrow
    \(F((X,\lambda))=\Hils(X,\lambda)\) in
    \(\ccorr(\Cst(G,\alpha),\Cst(H,\beta))\) (Data (ii) in
    Definition~\ref{def:bifunctor}).
  \item[2-arrow:] map a 2\nb-arrow t in
    \(\tcorr((G,\alpha),(H,\beta))\) to the isomorphism of
    \(\Cst\)\nb-correspondences \(F(\textup{t})\defeq\textup{T}\) in
    \(\ccorr(\Cst(G,\alpha),\Cst(H,\beta)\) as in
    Proposition~\ref{prop:vertical-functoriality}. Note that \(F\) is
    a functor from \(\tcorr((G,\alpha),(H,\beta))\) to
    \(\ccorr(\Cst(G,\alpha),\Cst(H,\beta))\); this follows from
    Corollary~\ref{cor:functoriality-on-vetical-arrows} and the remark
    above it (Data (ii) in Definition~\ref{def:bifunctor}).
  \item[Identity 2-morphism:] In \(\tcorr\) the identity 1\nb-arrow at
    \((G,\alpha)\) is \((G,\alpha\inverse)\), and the one at
    \(\Cst(G,\alpha)\) in \(\ccorr\) is \(\Cst(G,\alpha)\). The
    identity 2\nb-arrow \(\Cst(G,\alpha)\to \Hils(G,\alpha\inverse)\)
    is the identity isomorphism \(\Id_{\Cst(G,\alpha)}\) of
    \(\Cst\)\nb-correspondences as discussed in
    Example~\ref{exa:id-corr-iso} (Data (iii) in
    Definition~\ref{def:bifunctor}).
  \item[Natural transformation between composites:] Let
    \((X,\lambda) \colon (G,\alpha)\to (H,\beta)\) and \\
    \((Y,\mu) \colon (H,\beta) \to (K,\nu)\)
    be 1\nb-arrows. Then the natural transformation
    \(\phi((G,\alpha),(H,\beta),(K,\nu))\) between composites in
    \(\ccorr\) and \(\tcorr\) is the isomorphism
    \(\Hils(X,\lambda)\otimes_{\Cst(H,\beta)}\Hils(Y,\mu)\to
    \Hils(X\circ Y,\lambda\circ \mu)\) defined in
    Theorem~\ref{thm:well-behaviour-of-composition}.
  \end{description}
   
  We briefly recall definition of
  \(\phi((G,\alpha),(H,\beta),(K,\nu))\). Write simply \(\phi\). For
  \(f\in\Contc (X)\) and \(g\in \Contc(Y)\),
  \(\phi(f\otimes g)\in\Contc((X\times_{\base[H]}Y)/H)\) is the
  function
  
  \begin{multline}\label{eq:def-nat-trans}
    \phi(f\otimes g)([x,y])\\=\int_H f(x\eta) g(\eta\inverse y)
    b^{-1/2}(x\eta,\eta\inverse y)\,\dd\beta^{r_Y(y)}(\eta)\defeq
    [(f\otimes g)b^{1/2}]([x,y])
  \end{multline}
  where \([x,y]\) is the equivalence class of
  \((x,y)\in X\times_{\base[H]}Y\) in \((X\times_{\base[H]}Y)/H\), and
  \(b\) is the 0\nb-cocyle used to construct the family of measures on
  \((X\times_{\base[H]}Y)/H\). For more details, reader may refer to
  proof of Theorem~3.14 in~\cite{Holkar2017Composition-of-Corr},
  particularly, page~110 there.

  Now we prove that the pair \((F,\phi)= \bifunct\) is a morphism from
  the bicategory \(\tcorr\) to the bicategory \(\ccorr\). For this, we
  need to check that associativity coherence for transformations
  (Figures~\ref{fig:actual-pentagon}) and the coherence of (left and
  right) identity (Figure~\ref{fig:right-identity-coherence})
  hold. For this purpose, one can check the commutativity on the
  elementary tensors in the dense pre-Hilbert \(\Cst\)\nb-modules
  consisting of \(\Contc\) functions. Then the result can be extended
  to Hilbert \(\Cst\)\nb-modules by using the linearly of maps and
  standard density arguments.

  Checking commutativity of Figure~\ref{fig:actual-pentagon} is not so
  hard; it follows from a direct computation the definition of the
  natural ho \(\phi\) and the associativity of \(1\)\nb-arrows in
  \(\ccorr\). Let \(i\in \{1,2,3\}\), and let
  \((X_i,\lambda_i)\colon (G_i,\alpha_i)\to (G_{i+1},\alpha_{i+1})\)
  be three 1\nb-arrows in \(\ccorr\). Let
  \(f_i\in\Contc(X_i,\lambda_i)\) where \(i=1,2,3\). If one starts
  tracing how the function \((f_1\otimes f_2)\otimes f_3\) travels
  across Figure~\ref{fig:actual-pentagon}, we get
  Figure~\ref{fig:tracing-pentagon} from which, one can see the
  desired results clearly holds. In this figure, we write \(F(X_i)\)
  instead of \(F((X_i,\lambda_i))\) for simplicity of writing where
  \(i=1,2\).
  \begin{figure}[ht]
    \centering
    \begin{tikzpicture}[scale=2]
      \draw[dar] (0,0.9)--(0,0.1);
      \draw[dar] (0,1.9)--(0,1.1);
      \draw[dar] (2.5,2)--(0.5,2); 
      \draw[dar] (3,1.9)--(3,1.1);
      \draw[dar] (3, 0.9)--(3,0.1);
      \draw[dar] (2.25,0)--(0.7, 0);
      \node at (0,0
      )[scale=0.8]{\([f_1\otimes [f_2\otimes f_3]]
        \,b_1^{1/2}b_2^{1/2}\)}; \node at
      (0,1)[scale=0.8]{\(f_1\otimes ([f_2\circ f_3]\, b_2^{1/2})\)};
      \node at (0,2)[scale=0.8]{\(f_1\otimes (f_2\otimes f_3)\)};
      \node at (3,2)[scale=0.8]{\((f_1\otimes f_2)\otimes f_3\)};
      \node at
      (3,1)[scale=0.8]{\(([f_1\otimes f_2]\, b_1^{1/2})\otimes f_3\)};
      \node at
      (3,0)[scale=0.8]{\([[f_1\otimes f_2]\otimes f_3]\,
        b_1^{1/2}b_2^{1/2}\)};
      \node at (-0.6,0.5 )[scale=0.8]{\(\phi(X_1,X_2\circ X_3)\)};
      \node at
      (-0.6,1.5)[scale=0.8]{\(\textup{Id}_{X_1}\circ \phi(X_2,X_3)\)};
      \node at (1.5,2.1 )[scale=0.8]{\(a(F(X_1),F(X_2),F(X_3))\)};
      \node at (1.5,1.9 )[scale=0.8]{\(\sim\)}; \node at (3.7,1.5
      )[scale=0.8]{\(\phi(X_1,X_2)\circ\textup{Id}_{F(X_3)}\)}; \node
      at (3.6,0.5 )[scale=0.8]{\(\phi(X_1\circ X_2, X_3)\)}; \node at
      (1.5,-0.15 )[scale=0.8]{\(F(a(X_1, X_2, X_3))\)}; \node at
      (1.5,0.1 )[scale=0.8]{\(\sim\)};
    \end{tikzpicture}
    \caption{}
    \label{fig:tracing-pentagon}
  \end{figure}
  
  \noindent \emph{The identity isomorphisms:} Let \((X,\lambda)\) be a
  correspondence from \((G,\alpha)\) to \((H,\beta)\). First we check
  the left identity coherence. Let \((G\circ X,\mu)\) denote a
  composite of the identity correspondence \((G,\alpha\inverse)\) at
  \((G,\alpha)\) and \((X,\lambda)\). Let \(b\) be a 0\nb-cocycle on
  the transformation groupoid \(Q=(G\ltimes_{\base[G]}X)\rtimes G\)
  that is used to create \(\mu\). Let \(\Delta_1\) denote the
  1\nb-cocycle on \(Q\) such that \(d^0(b)=\Delta_1\). For the left
  identity coherence we need to check that
  Figure~\ref{fig:actual-right-id-coh} commutes.
  \begin{figure}[ht]
    \[
      \begin{tikzcd}[column sep=2.7cm]
        \Hils(X,\lambda) & \Hils(G\circ X,\mu')
        \ar[l,"\textup{T}"', Rightarrow]\\
        \Cst(G,\alpha)\otimes \Hils(X,\lambda)
        \ar[u,"\textup{i}",Rightarrow] \ar[r,
        "\Id_{\Cst(G,\alpha)}\otimes
        \Id_{\Hils(X,\lambda)}"',Rightarrow]&
        \Hils(G,\alpha\inverse)\otimes\Hils(X,\lambda)\ar[u,
        "{\phi(G,X)}"',Rightarrow]
      \end{tikzcd}
    \]
    \caption{}
    \label{fig:actual-right-id-coh}
  \end{figure}
  In this figure, the map \(\Id_{\Cst(G,\alpha)}\) in the bottom
  horizontal arrow the identity isomorphism discussed in
  Example~\ref{exa:id-corr}; this isomorphism of Hilbert
  \(\Cst(G,\alpha)\)\nb-modules is induced by the identity map on
  \(\Contc(G)\). Thus
  \(\Id_{\Cst(G,\alpha)}\otimes\Id_{\Hils(X,\lambda)}\) is the
  identity map. The right vertical map is the assignment \(\phi\) of
  2\nb-arrows. The top horizontal map \textup{T} is given in
  Proposition~\ref{prop:vertical-functoriality}. And i is the
  isomorphism \(a\otimes b\mapsto ab\) for an elementary tensor
  \(a\otimes b\in \Cst(G,\alpha)\otimes \Hils(X,\lambda)\). Let
  \(f\in\Contc(G)\subseteq \Cst(G,\alpha)\) and
  \(g\in\Contc(X)\subseteq \Hils(X,\lambda)\). Then for the elementary
  tensor \(f\otimes g\in \Cst(G,\alpha)\otimes \Hils(X,\lambda)\),
  \[
    \phi(G,X)\left(\Id_{\Cst(G,\alpha)}\otimes \Id_{\Hils(X,\lambda)}
      \left(f\otimes g\right)\right) = [(f\otimes g)b^{1/2}]
  \]
  where \((f\otimes g)\) is defined in
  Equation~\eqref{eq:def-nat-trans}. Recall from
  Equation~\ref{eq:left-id-corr-2} and Example~\ref{exa:id-corr} that
  \(\dd\mu/\dd\lambda (\gamma\inverse, x)=b(s_G(\gamma),\gamma\inverse
  x)\) for \((\gamma\inverse,x)\in G\ltimes X\). Also, recall that in
  this Example, we identify \((G\times_{\base[G]}X)/G\) by \(G\)
  itself; this allows us to consider \(\mu\) as a family of measures
  on \(X\). With this in mind, we write
  \begin{align*}
    \textup{T}( [(f\otimes g)b^{1/2}])(x) =& [(f\otimes
                                             g)b^{1/2}][r_X(x),x]\cdot { \frac{ \dd \mu_u }{ \dd \lambda_u }([r_X(x), x] ) }^{1/2}\\
    =& b^{1/2}(r_X(x),x) \int_G f(\gamma) g(\gamma\inverse x)
       b^{-1/2}(\gamma,\gamma\inverse x)\,\dd\alpha^{r_X(x)}(\gamma)\\
    =& \int_G f(\gamma) g(\gamma\inverse x)
       \sqrt{\frac{b(r_X(x),x)}{b(\gamma,\gamma\inverse
       x)}}\,\dd\alpha^{r_X(x)}(\gamma).
  \end{align*}
  Note that \(r_X(x)=r_G(\gamma)\), and that for
  \(\gamma,\gamma\inverse x,\gamma\inverse)\in Q\),
  \(s_Q(\gamma,\gamma\inverse
  x,\gamma\inverse)=(r_G(\gamma),\gamma\inverse
  x)=(r_X(x),\gamma\inverse x)\) and
  \(r_Q)(\gamma,\gamma\inverse
  x,\gamma\inverse)=(\gamma,\gamma\inverse x)\). With this we re-write
  the last term in above computation as follows and compute further:
  \begin{align*}
    & \int_G f(\gamma) g(\gamma\inverse x)
      \sqrt{\frac{b\circ s_Q(\gamma,\gamma\inverse
      x,\gamma\inverse)}{b\circ r_Q(\gamma,\gamma\inverse
      x,\gamma\inverse)}}\,\dd\alpha^{r_X(x)}(\gamma)\\
    =& \int_G f(\gamma) g(\gamma\inverse x)
       \sqrt{\Delta_{1}(\gamma,\gamma\inverse
       x,\gamma\inverse)}\,\dd\alpha^{r_X(x)}(\gamma).
  \end{align*}
  But
  \(\Delta_1(\gamma,\gamma\inverse
  x,\gamma\inverse)=\Delta(\gamma,\gamma\inverse x)\). Therefore, the
  above term equals
  \begin{equation*}
    \int_G f(\gamma) g(\gamma\inverse x) \Delta^{1/2}(\gamma,\gamma\inverse
    x)\,\dd\alpha^{r_X(x)}(\gamma)= f\cdot g(x) = \textup{i}(f\otimes
    g)(x)
  \end{equation*}
  where \(f\cdot g\) is the action of \(f\in \Contc(G)\) on
  \(g\in \Contc(X)\). This shows that
  Figure~\ref{fig:actual-right-id-coh} commutes.
  
  With the help of Example~\ref{exa:id-corr}, one may verify the
  coherence for the right identity in a similar fashion as above.
\end{proof}

Finally, we give three illustrations describing the
\(\Cst\)\nb-(bi)functor: the first one involving spaces, second one
involving topological quivers and the last one involving topological
groups. We avoid the rigorous definition of the term embedding below.

\begin{illustration}[The \(\Cst\)\nb-functor on spaces]
  Let \(\mathfrak{S}\) denote the category of locally compact spaces
  with continuous functions a morphisms. Considering a space as a
  trivial ({\'e}tale) groupoid, a continuous map \(f\colon X\to Y\) of
  spaces gives us the topological correspondence
  \((Y,\delta_X)\colon X\to Y\), see Example~3.1
  in~\cite{Holkar2017Construction-of-Corr} for details. If \(g\colon
  Y\to Z\) is another map of spaces then Example~4.1
  in~\cite{Holkar2017Composition-of-Corr} shows that the composition
  \[
(Y,\delta_X)\circ (Z,\delta_Y)=(Z,\delta_X)
\]
where \((Z,\delta_X)\) is the correspondence associated with
\(g\circ f\colon X\to Z\). This \emph{embeds}~\(\mathfrak{S}\) in
\(\tcorr\). Theorem~\ref{thm:functoriality} is clear for
\(\mathfrak{S}\) due to~\cite{Holkar2017Composition-of-Corr}{Example
  4.1} --- this is a halve of the well-known functoriality in Gelfand's
characterisation of abelian \(\Cst\)\nb-algebras.
\end{illustration}

\begin{illustration}[The (bi)category of quivers]
  Recall
  from~\cite{Muhly-Tomforde-2005-Topological-quivers}*{Definition
    3.17} that a topological quiver from a space \(X\) to a space
  \(Y\) is quintuple \((Z,b,f,\lambda)\) where
  \(X\xleftarrow{b} Z\xrightarrow{f} Y\) are continuous maps and
  \(\lambda\) is a continuous family of measures along \(f\). As shown
  in~\cite{Holkar2017Construction-of-Corr}{Examples~3.3} a topological
  quiver is a topological correspondence. The composite of two
  topological quivers is again a topological quiver as explained
  in~\cite{Holkar2017Composition-of-Corr}*{Example~4.2}. Thus,
  topological quivers form a category which clearly embeds in
  \(\tcorr\). One can extend this category to a bicategory by adding
  vertical arrows between the spaces involved in the quivers; these
  two arrows are homeomorphisms commuting with the maps \(b\) and
  \(f\). Then Theorem~\ref{thm:functoriality} explains the
  \(\Cst\)\nb-functor for quivers--- this fuctoriality can be checked
  more easily for quivers.
\end{illustration}

\begin{illustration}[The \(\Cst\)\nb-functor for groups]
  As in~\cite{Holkar2017Construction-of-Corr}*{Examples~3.4}, a group
  homomorphisms \(\phi\colon G\to H\) of locally compact groups yields
  the topological correspondence
  \((H,\beta\inverse)\colon (G,\alpha)\to (H,\beta)\) where \(\alpha\)
  and \(\beta\) are the Haar measures on \(G\) and \(H\),
  respectively. Example~4.3 in~\cite{Holkar2017Composition-of-Corr}
  shows that a homomorphism going to a correspondence is
  functorial. This embeds the category of locally compact groups into
  \(\tcorr\). The \(\Cst\)\nb-functor assigns a group \(G\) its
  \(\Cst\)\nb-algebra, a group homomorphism \(\phi\colon G\to H\) the
  \(\Cst\)\nb-correspondence \(\Cst(H,\beta\inverse)\).
\end{illustration}

\noindent\textbf{Acknowledgement:} The author is thankful to Ralf
Meyer and Suliman Albandik for many fruitful discussions.

\end{document}